\documentclass[11pt]{amsart}
\usepackage{amssymb,amsmath,amsthm, longtable}
\usepackage{enumerate, url, verbatim}
\usepackage{comment}
\usepackage{xy}
\xyoption{all}
\newcommand{\leg}[2]{\genfrac{(}{)}{}{}{#1}{#2}}

\newtheorem{theorem}{Theorem}
\newtheorem{lemma}[theorem]{Lemma}
\newtheorem{corollary}[theorem]{Corollary}

\newtheorem{proposition}[theorem]{Proposition}

\theoremstyle{remark}
\newtheorem{definition}[theorem]{Definition}
\newtheorem{remarks}[theorem]{Remarks}

\newtheorem*{remark}{Remark}

\numberwithin{theorem}{section} \numberwithin{equation}{section}

\newcommand{\al}{\alpha}
\newcommand{\be}{\beta}
\newcommand{\ga}{\gamma}
\renewcommand{\th}{\theta}
\newcommand{\eps}{\varepsilon}
\newcommand{\LR}{\longrightarrow}
\newcommand{\gd}{{\mathfrak d}}

\newcommand{\mfP}{\mathfrak{P}}

\newcommand{\mfa}{\mathfrak{a}}
\newcommand{\mfd}{\mathfrak{d}}
\newcommand{\ov}[1]{\overline{#1}}

\newcommand{\isom}{\simeq}

\newcommand{\calD}{\mathcal{D}}
\newcommand{\Gd}{{\mathfrak D}}

\newcommand{\FF}{\mathcal{F}}

\newcommand{\om}{\omega}
\newcommand{\mfp}{\mathfrak{p}}
\newcommand{\mfb}{\mathfrak{b}}

\newcommand{\mfc}{\mathfrak{c}}

\newcommand{\calL}{\mathcal{L}}
\newcommand{\calA}{\mathcal{A}}

\newcommand{\calF}{\mathcal{F}}

\newcommand{\Cl}{{\text {\rm Cl}}}
\newcommand{\Tr}{{\text {\rm Tr}}}
\newcommand{\rk}{{\text {\rm rk}}}
\newcommand{\sign}{{\text {\rm sign}}}

\newcommand{\WL}{\widetilde{L}}
\newcommand{\Q}{\mathbb{Q}}

\newcommand{\Z}{\mathbb{Z}}

\newcommand{\LL}{{\mathcal L}}
\newcommand{\p}{\mathfrak p}
\renewcommand{\a}{\mathfrak a}
\newcommand{\mfm}{\mathfrak m}
\newcommand{\q}{\mathfrak q}
\renewcommand{\c}{\mathfrak c}
\renewcommand{\d}{\mathfrak d}
\newcommand{\N}{{\mathcal N}}
\DeclareMathOperator{\NO}{\mathcal{N}}
\newcommand{\textmod}{{\text {\rm mod}}}

\newcommand{\Gal}{{\text {\rm Gal}}}

\newcommand{\Disc}{\textnormal{Disc}}

\newcommand{\Ker}{\textnormal{Ker}}

\renewcommand{\b}{{\mathfrak b}}
\newcommand{\GP}{{\mathfrak P}}

\newcommand{\z}{\zeta}

\begin{document}
\title[Dirichlet Series Associated to Quartic Fields with Given Resolvent]
{Dirichlet Series Associated to Quartic Fields with Given Resolvent}

\author{Henri Cohen}
\address{Universit\'e Bordeaux I, Institut de Math\'ematiques, U.M.R. 5251 du
C.N.R.S, 351 Cours de la Lib\'eration, 33405 TALENCE Cedex, FRANCE}
\email{Henri.Cohen@math.u-bordeaux1.fr}
\author{Frank Thorne}
\address{Department of Mathematics, University of South Carolina, 1523 Greene Street, Columbia, SC 29208, USA}
\email{thorne@math.sc.edu}


\begin{abstract}

Let $k$ be a cubic field. We give an explicit formula for the Dirichlet 
series $\sum_K|\Disc(K)|^{-s}$, where the sum is over isomorphism
classes of all quartic fields whose cubic resolvent field is isomorphic to 
$k$. Our work is a sequel to the unpublished preprint \cite{CDO_quartic} whose
results have been summarized in \cite{Coh_a4s4}, so we include complete
proofs so as not to rely on unpublished work.

This is a companion paper to \cite{CT3} where we compute the
Dirichlet series associated to cubic fields having a given quadratic
resolvent.
\end{abstract}

\maketitle
\section{Introduction}
In a previous paper \cite{CT3}, we studied the problem of enumerating cubic 
fields\footnote{Note that in this paper, number fields are always considered 
up to isomorphism.} with fixed quadratic resolvent.
A classical result of Cohn \cite{cohn} is that
\begin{equation}\label{eqn_cohn}
\sum_{K \ \textnormal{cyclic cubic}} \frac{1}{\Disc(K)^s} = - \frac{1}{2} + \frac{1}{2} \bigg( 1 + \frac{1}{3^{4s}} \bigg) \prod_{p \equiv 1 \pmod 6}
\bigg(1 + \frac{2}{p^{2s}}\bigg)\;.
\end{equation}
We generalized this as follows. If $K$ is a non-cyclic cubic field,
then its Galois closure $\widetilde{K}$ contains a unique quadratic field $k$, called the \emph{quadratic resolvent}.
We have $\Disc(K) = \Disc(k) f(K)^2$ for an integer $f(K)$, and for each fixed $k$ we proved explicit formulas for the Dirichlet
series $\sum_K f(K)^{-s}$, where the sum is over all cubic fields $K$ with quadratic resolvent $k$. For example, if
$k = \Q(\sqrt{-255})$ we have
\begin{equation}\label{eqn_cub_ex}
\sum_K \dfrac{1}{f(K)^s}
=-\frac{1}{2} + \dfrac{1}{2}\left(1+\dfrac{2}{3^s}+\dfrac{6}{3^{2s}}\right)\prod_{\leg{6885}{p}=1}\left(1+\dfrac{2}{p^s}\right)\\
+\left(1-\dfrac{1}{3^s}\right)\prod_p\left(1+\dfrac{\om_L(p)}{p^s}\right)\;,\end{equation}
where the sum is over all cubic fields with quadratic resolvent $k$, $L$ is 
the cubic field of discriminant $6885=(-27)\cdot(-255)$ determined by 
$x^3 - 12x - 1 = 0$, and
$\om_{L}(p)$ is equal to $2$ or $-1$ when $p$ is totally split or inert in $L$
respectively, and $\om_{L}(p) = 0$ otherwise. In general, the sum has a main 
term plus one additional term for each cubic field of discriminant 
$-D/3$, $-3D$, or $-27D$, where $D$ is the discriminant of $k$.

Our work extended work of the first author and Morra \cite{CM}, which 
established more general formulas in a less explicit form. In the quartic 
case, such formulas have been proved by the first author, Diaz y Diaz, and 
Olivier \cite{Coh_a4s4, CDO_C4, CDO_D4, CDO_V4, CDO_quartic}. This work also
yields explicit Dirichlet series similar to \eqref{eqn_cohn} and
\eqref{eqn_cub_ex}. For example, for any \emph{Abelian} group $G$ set
\begin{equation}
\Phi(G,s)=\sum_{\Gal(K/\Q)\isom G}\dfrac{1}{|\Disc(K)|^s}\;.
\end{equation}
Then we have the explicit Dirichlet series 
$$\Phi(C_2,s)=\left(1-\dfrac{1}{2^s}+\dfrac{2}{2^{2s}}\right)\dfrac{\zeta(s)}{\zeta(2s)}-1\;,$$
\begin{align*}
\Phi(C_4,s)&=\dfrac{\zeta(2s)}{2\zeta(4s)}\biggl(\biggl(1-\dfrac{1}{2^{2s}}+\dfrac{2}{2^{4s}}+\dfrac{4}{2^{11s}+2^{9s}}\biggr)\prod_{p\equiv1\pmod4}\biggl(1+\dfrac{2}{p^{3s}+p^s}\biggr)\\
&\kern60pt-\biggl(1-\dfrac{1}{2^{2s}}+\dfrac{2}{2^{4s}}\biggr)\biggr)\;,
\end{align*}
\begin{align*}
\Phi(V_4,s)&=\dfrac{1}{6}\biggl(1+\dfrac{3}{2^{4s}}+\dfrac{6}{2^{6s}}+\dfrac{6}{2^{8s}}\biggr)\prod_{p \neq 2}\biggl(1+\dfrac{3}{p^{2s}}\biggr)-\dfrac12\Phi(C_2,2s)-\dfrac16\;.
\end{align*}

The same authors proved similar formulas for those $C_4$ and $V_4$ extensions
having a fixed quadratic subfield; the former is Theorem 4.3 of \cite{CDO_C4} and the latter
is unpublished. They also obtained analogous formulas for $D_4$ extensions, for which 
we refer to \cite{CDO_D4} and Section 7.1 of \cite{CDO_all}.

In the present paper we tackle this problem for $A_4$ and $S_4$-quartic fields. We count 
such fields by their \emph{cubic resolvents}:
Suppose that $K/\Q$ is a quartic field whose Galois closure 
$\widetilde{K}$ has Galois group $A_4$ or $S_4$. In the $A_4$ case, 
$\widetilde{K}$ contains a unique cyclic cubic subfield $k$, and in the $S_4$
case, $\widetilde{K}$ contains three isomorphic noncyclic cubic subfields $k$.
In either case $k$ is called the \emph{cubic resolvent} of $K$, it is 
unique up to isomorphism, and it satisfies $\Disc(K) = \Disc(k) f(K)^2$ for
some integer $f(K)$. 


Let $\calF(k)$ be the set of all $A_4$ or $S_4$-quartic fields whose cubic
resolvent is isomorphic to $k$. We set the following definition.
\begin{definition}\label{defphik} For a cubic field $k$, we set
$$\Phi_k(s)=\dfrac{1}{a(k)}+\sum_{K\in\FF(k)}\dfrac{1}{f(K)^s}\;,$$
where $a(k)=3$ if $k$ is cyclic and $a(k)=1$ otherwise.\footnote{This differs 
slightly from the definition given in \cite{Coh_a4s4}.}

\end{definition}
We will prove explicit formulas for $\Phi_k(s)$, building on previous work of the first author, 
Diaz y Diaz, and Olivier (\cite{CDO_quartic}; see also \cite{Coh_a4s4} for a 
published summary) which, like the subsequent paper \cite{CM}, established a more general
but less explicit formula.
Since \cite{CDO_quartic} is unpublished, we will include complete proofs of the results we need.

In the cubic case, our formulas involved sums over cubic fields of discriminant $-D/3$, $-3D$, and $-27D$, and in the quartic case we will sum over fields in
a similar set $\LL_2(k)$:

\begin{definition}\label{defll} Given any cubic field $k$ (cyclic or not),
let $\LL(k)$ be the set of isomorphism classes of quartic fields whose
cubic resolvent is isomorphic to $k$, with the additional restriction
that the quartic is totally real when $k$ is such. Furthermore, for any $n$
define $\LL(k,n^2)$ to be the subset of $\LL(k)$ of those fields with 
discriminant equal to $n^2\Disc(k)$.

Finally, we define $\LL_{tr}(k,64)$ to be the subset of those $L\in\LL(k,64)$ 
such that $2$ is totally ramified in $L$, and we set
$$\LL_2(k)=\LL(k,1)\cup\LL(k,4)\cup\LL(k,16)\cup\LL_{tr}(k,64)\;.$$
\end{definition}

Note that if $k$ is totally real the elements of $\FF(k)$ are totally real or 
totally complex, and $\LL(k)$ is the subset of totally real ones, while if 
$k$ is complex then the elements of $\LL(k)=\FF(k)$ have mixed signature 
$r_1=2$, $r_2=1$.

\begin{remark} In this paper, quartic fields with cubic resolvent $k$
will be denoted $K$ or $L$. Generally $K$ will refer to fields enumerated
by $\Phi_k(s)$, and $L$ will refer to fields in $\calL_2(k)$. Note however 
that in many places this distinction will be irrelevant.
\end{remark}

We introduce some standard notation for splitting types of primes in a 
number field. If $L$ is, say, a quartic field, and $p$ is a prime for which 
$(p) = \mfp_1^2 \mfp_2$ in $L$, where $\mfp_i$ has residue class degree $i$ 
for $i = 1, 2$, we say that $p$ has splitting type $(21^2)$ in $L$ (or simply
that $p$ is $(21^2)$ in $L$). Other splitting types such as 
$(22), \ (1111), \ (1^4)$, etc. are defined similarly.
Moreover, when $2$ has type $(1^2 1)$ in a cubic field $k$,
we say that $2$ has type $(1^21)_0$ or $(1^21)_4$ depending on whether 
$\Disc(k) \equiv 0 \pmod 8$ or $\Disc(k) \equiv 4 \pmod 8.$

\begin{definition}\label{def_omega}
Let $L$ be a quartic $A_4$ or $S_4$-quartic field. For a prime number $p$ we 
set
$$\om_L(p)=\begin{cases}
-1&\text{\quad if $p$ is $(4)$, $(22)$, or $(21^2)$ in $L$\;,}\\
1&\text{\quad if $p$ is $(211)$ or $(1^211)$ in $L$\;,}\\
3&\text{\quad if $p$ is $(1111)$ in $L$\;,}\\
0&\text{\quad otherwise\;.}
\end{cases}$$
\end{definition}

Our main theorems are the following:

\begin{theorem}\label{thm_main} Let $k$ be a cubic field, cyclic or not.
\begin{enumerate}\item We have
\begin{align*}2^{r_2(k)}\Phi_k(s)&=\dfrac{1}{a(k)}M_1(s)\prod_{p\Z_k=\p_1\p_2,\ p\ne2}\left(1+\dfrac{1}{p^s}\right)\prod_{p\Z_k=\p_1^2\p_2,\ p\ne2}\left(1+\dfrac{1}{p^s}\right)
\prod_{p\Z_k=\p_1\p_2\p_3,\ p\ne2}\left(1+\dfrac{3}{p^s}\right)\\
&\phantom{=}+\sum_{L\in\LL_2(k)}M_{2,L}(s)\prod_{p\ne2}\left(1+\dfrac{\om_L(p)}{p^s}\right)\;,\end{align*}
where the $2$-Euler factors $M_1(s)$ and $M_{2,L}(s)$ are given
in the tables below, where the $k$ and $L$-splits refer to the splitting of
the prime $2$.
\item If $k$ is cyclic we have
\begin{equation}\Phi_k(s)=\dfrac{1}{3}M_1(s)\prod_{p\Z_k=\p_1\p_2\p_3,\ p\ne2}\left(1+\dfrac{3}{p^s}\right)
+\sum_{L\in\LL(k,1)}M_{2,L}(s)\prod_{p\Z_k=\p_1\p_2\p_3,\ p\ne2}\left(1+\dfrac{\om_L(p)}{p^s}\right)\;.\end{equation}
\end{enumerate}
\end{theorem}

\bigskip

\centerline{
\begin{tabular}{|c||c|c|}
\hline
$k$-split & $M_1(s)$ & $8M_1(1)$ \\
\hline\hline
$(3)$ & $1+3/2^{3s}$ & $11$ \\
\hline
$(21)$ & $1+1/2^{2s}+4/2^{3s}+2/2^{4s}$ & $15$ \\
\hline
$(111)$ & $1+3/2^{2s}+6/2^{3s}+6/2^{4s}$ & $23$ \\
\hline
$(1^21)_0$ & $1+1/2^s+2/2^{3s}+4/2^{4s}$ & $16$ \\
\hline
$(1^21)_4$ & $1+1/2^s+2/2^{2s}+4/2^{4s}$ & $18$ \\
\hline
$(1^3)$ & $1+1/2^s+2/2^{3s}$ & $14$ \\
\hline
\end{tabular}}

\bigskip

\centerline{
\begin{tabular}{|c|c|c||c|||c|c|c||c|}
\hline
$k$-split & $L$-split & $n^2$ & $M_{2,L}(s),\ L\in\LL(k,n^2)$ & $k$-split & $L$-split & $n^2$ & $M_{2,L}(s),\ L\in\LL(k,n^2)$ \\
\hline\hline
$(3)$ & $(31)$ & $1$ & $1+3/2^{3s}$ & $(1^21)_0$& $(21^2)$ & $1$ & $1+1/2^s+2/2^{3s}-4/2^{4s}$ \\
\hline
$(3)$ & $(1^4)$ & $64$ & $1-1/2^{3s}$ & $(1^21)_0$& $(1^211)$ & $1$ & $1+1/2^s+2/2^{3s}+4/2^{4s}$ \\
\hline
$(21)$& $(4)$ & $1$ & $1+1/2^{2s}-2/2^{4s}$ & $(1^21)_0$& $(1^2 1^2)$ & $4$ & $1+1/2^s-2/2^{3s}$ \\
\hline
$(21)$& $(211)$ & $1$ & $1+1/2^{2s}+4/2^{3s}+2/2^{4s}$ & $(1^21)_0$& $(1^4)$ & $64$ & $1-1/2^s$ \\
\hline
$(21)$& $(2^2)$ & $16$ & $1+1/2^{2s}-4/2^{3s}+2/2^{4s}$ & $(1^21)_4$& $(21^2)$ & $1$ & $1+1/2^s+2/2^{2s}-4/2^{4s}$ \\
\hline
$(21)$& $(1^2 1^2)$ & $16$ & $1+1/2^{2s}-2/2^{4s}$ & $(1^21)_4$& $(1^211)$ & $1$ & $1+1/2^s+2/2^{2s}+4/2^{4s}$ \\
\hline
$(21)$& $(1^4)$ & $64$ & $1-1/2^{2s}$ & $(1^21)_4$& $(2^2)$ & $4$ & $1+1/2^s-2/2^{2s}$ \\
\hline
$(111)$& $(22)$ & $1$ & $1+3/2^{2s}-2/2^{3s}-2/2^{4s}$ & $(1^21)_4$& $(2^2)$ & $16$ & $1-1/2^s$ \\
\hline
$(111)$& $(2^2)$ & $16$ & $1-1/2^{2s}-2/2^{3s}+2/2^{4s}$ & $(1^21)_4$& $(1^2 1^2)$ & $16$ & $1-1/2^s$ \\
\hline
$(111)$& $(1111)$ & $1$ & $1+3/2^{2s}+6/2^{3s}+6/2^{4s}$ & $(1^3)$& $(1^3 1)$ & $1$ & $1+1/2^s+2/2^{3s}$ \\
\hline
$(111)$& $(1^2 1^2)$ & $16$ & $1-1/2^{2s}+2/2^{3s}-2/2^{4s}$ & $(1^3)$& $(1^4)$ & $4$ & $1+1/2^s-2/2^{3s}$ \\
\hline
&&&&$(1^3)$& $(1^4)$ & $64$ & $1-1/2^s$ \\
\hline
\end{tabular}}

\bigskip

\begin{remarks}\begin{itemize}
\item It will follow from
Proposition \ref{prop_l2k} \eqref{it_l2k_empty} that (2) is a special case of (1), and we will prove the two results
simultaneously.
\item In case (2) where $k$ is cyclic, the splitting behavior of the primes in
$k$ is determined by congruence conditions. Also, since $2$ can only split
as $(3)$ or $(111)$ in $k$, only $\LL(k,1)$ occurs and the list of possible
$2$-Euler factors is very short: only two cases for $M_1(s)$ and
three cases for $M_{2,L}(s)$.
\item As a check on our results, we numerically verified the above theorems
for the first $10000$ totally real and the first $10000$ complex cubic fields.
As a further check on the consistency of our results, we also observe that the
values $8M_1(1)$ are equal to the constants $c_2(k)$ in Theorem 1.2 of
\cite{Coh_a4s4}.\end{itemize}
\end{remarks}

The formulas are much longer in the $S_4$ case than in the $S_3$ case, simply 
because the number of splitting types in an $S_4$-quartic 
field is much larger than in a cubic field. However,
studying quartic extensions having a given cubic resolvent is in fact
\emph{simpler} than the analogous study of cubic extensions having a given
quadratic resolvent. The reason for this is as follows: in the cubic case, it
is necessary to count cyclic cubic extensions of a quadratic field, and for 
this (because of Kummer theory, or equivalently, of class field theory) we 
must adjoin cube roots of unity, which complicates matters. On the other hand,
we will see that in the quartic $A_4$ and $S_4$ cases we must count $V_4$ 
extensions of a cubic field having certain properties, and here it is 
\emph{not} necessary to adjoin any root of unity since the square roots of 
unity are already in the base field.
\\
\\
{\bf Outline of the paper.}
We begin in Section \ref{sec_quartic} by recalling a parametrization
of quartic fields $K$ in term of pairs $(k, K_6)$, where $k$ is the cubic
resolvent of $K$ and $K_6$ is a quadratic extension of `trivial norm'. This allows us 
to count quartic fields by counting such quadratic extensions.

Section \ref{sec:CDO} consists essentially of work of the first author,
Diaz y Diaz, and Olivier \cite{CDO_quartic}, which establishes a version of Theorem \ref{thm_main}
in an abstract setting. As this work was not published, we provide full proofs here.
In Section \ref{sec_c} we study certain groups $C_{\c^2}$ appearing in Section \ref{sec:CDO}, and prove that
they are essentially class groups.

In Section \ref{sec_xnew} we state a theorem establishing the possible splitting types of primes $p$
in quartic fields and their associated pairs $(k, K_6)$. As the proof requires lots of checking of special cases,
and also overlaps with unpublished work of Martinet \cite{Mar}, we only sketch the proof here, but a note
with complete proofs is available from the second author's website.

In Sections  
\ref{sec_arith_quartic} and \ref{sec_arith_s4} we further study the arithmetic of quartic fields associated
to characters of $C_{\c^2}$; some of these results are potentially of independent interest. This then brings us to the proofs of our main results in 
Section \ref{sec_quartic_proofs}. 
In Section \ref{sec_sig} we prove a version of our main theorem counting quartic fields with prescribed
signature conditions; the statement and proof of this generalization turn out to be surprisingly simple.`
We conclude 
in Section 
\ref{sec_computations} with numerical examples 
which were helpful in double-checking our results. 

\section*{Acknowledgments}
The authors would like to thank Karim Belabas, Arnaud Jehanne, Franz 
Lemmermeyer, Guillermo Mantilla-Soler, Jacques Martinet, and 
Simon Rubinstein-Salzedo, among many others, for helpful discussions related 
to the topics of this paper.

The second author would also like to thank the National Science Foundation\footnote{Grant DMS-1201330} for financial support.

\section{The Parametrization of Quartic Fields}\label{sec_quartic}

\begin{definition}\label{deftriv}\hfill \begin{enumerate}
\item We will say that an element $\al\in k^*$ (resp., an ideal $\a$ of $k$)
has square norm if $\NO(\al)$ (resp., $\N(\a)$) is a square in $\Q^*$.\footnote{Note that in \cite{Coh_a4s4} there is a misprint in the definition of square
norm, where ``$\N_{K_6/k}(\al)$ square in $k$'' should be replaced by what we 
have written, i.e., simply ``$\al$ of square norm'', in other words $\NO(\al)$ 
square in $\Q^*$.}

\item We will say that a quadratic extension 
$K_6/k$ has \emph{trivial norm} if there exists $\al\in k^*\setminus {k^*}^2$ 
of square norm such that $K_6=k(\sqrt{\al})$. (Observe for $k$ cubic that this implies $\alpha \not \in \Q$.)
\end{enumerate}\end{definition}

Note that if the principal ideal $(\al)$ has square
norm then $\al$ has either square norm or minus square norm, but
since we will only be considering such elements in \emph{cubic} fields,
this means that $\pm\al$ has square norm for a suitable sign.

It is fundamental to our efforts that quartic fields $K$ with cubic resolvent 
$k$ correspond to quadratic extensions $K_6/k$ of trivial norm.
We review this correspondence here.

\begin{theorem}\label{thm_quartic_corr} 
There is a correspondence between isomorphism classes of $A_4$ or
$S_4$-quartic fields $K$, and pairs $(k, K_6)$, where $k$ is the cubic 
resolvent field of $K$, and $K_6/k$ is a quadratic extension of trivial
norm. Under this correspondence we have $\Disc(K) = \Disc(k) \N(\mfd(K_6/k))$.

If $K$ is an $S_4$-field then this correspondence is a bijection, and $K_6$ is
equal to the unique extension of $k$ with $\Gal(\widetilde{K}/K_6) \simeq C_4$.
If $K$ is an $A_4$-field, then $k$ has three quadratic extensions, given by 
adjoining a root of $\al$ or either of its nontrivial conjugates, and this 
correspondence is $1$-to-$3$, with any of these fields yielding the same $K$
(up to isomorphism).
\end{theorem}

When we apply this theorem to quartic fields $L \in \LL_2(k)$ we still denote 
the corresponding sextic field by $K_6$.

\begin{remark}
Theorem \ref{thm_quartic_corr} has rough parallels in the theory of prehomogeneous vector spaces, for example in 
Bhargava's work on `higher composition laws' \cite{B2, B3}.
Roughly speaking, Bhargava proves that the sets $(R, I)$, where $R$ is a cubic ring and $I$ is an
ideal of $R$ whose square is principal, and $(Q, R)$, where $Q$ is a quartic ring and $R$ is a cubic resolvent ring of $Q$, are parameterized
by group actions on lattices which are $\Z$-dual to one another. 

The analogy is not exact, but as class field theory connects quadratic extensions $K_6/k$ to
index two subgroups of $\Cl(k)$, we can see a parallel to Bhargava's and related work.
\end{remark}

\begin{proof}
This is well known, but for the sake of completeness we sketch a proof.

For an $A_4$-quartic field $K$, $\widetilde{K}$ contains 
a unique $2$-Sylow subgroup, and therefore $\widetilde{K}$ contains a unique 
cubic subfield $k$, which must be cyclic. $\widetilde{K}$ also
contains three sextic fields; writing $K_6 = k(\sqrt{\al})$ for one of them,
the other two are $k(\sqrt{\al'})$ for the two conjugates $\al'$ of $\al$,
so $\widetilde{K}$ contains $\Q(\sqrt{\NO(\al)})$. However, since $A_4$ has
no subgroup of order $6$ this cannot be a quadratic extension, so $\al$ must
have square norm.

Conversely, given $k$ and $K_6 = k(\sqrt{\al})$, one obtains 
$\widetilde{K}$ by adjoining square roots of the conjugates of $\al$, and 
checks that $\widetilde{K}$ is Galois over $\Q$ with Galois group $A_4$. There
are four isomorphic quartic subextensions $K$, corresponding to the $3$-Sylow
subgroups of $A_4$. This proves the correspondence for $A_4$-extensions, and 
any of the quadratic extensions of $k$ produce the same quartic field $K$ (up 
to isomorphism).

For an $S_4$-quartic field $K$, $\widetilde{K}$ contains three conjugate 
$2$-Sylow subgroups, corresponding to three conjugate noncyclic cubic fields 
$k$, with Galois closure $k(\sqrt{D})$ (where $D := \Disc(k)$), 
$\Gal(\widetilde{K}/k(\sqrt{D})) \isom V_4=C_2\times C_2$, and 
$\Gal(\widetilde{K}/k) \isom D_4$. Since $D_4$ contains three subgroups of
size $4$, there exist three quadratic subextensions of $\widetilde{K}/k$, and 
since $S_4$ has a unique subgroup of order $12$, corresponding to
$\Q(\sqrt{D})$, and $\sqrt{D}\notin k$, these subextensions are $k(\sqrt{D})$,
$k(\sqrt{\al})$, and $k(\sqrt{\al D})$ for some $\al\in k^*$.

Now if we denote by $\al'$ and $\al''$ the nontrivial conjugates of $\al$
and by $k'$, $k''$ the corresponding conjugate fields of $k$, 
the fields $k'(\sqrt{\al'})$ and $k''(\sqrt{\al''})$ are in $\widetilde{K}$,
hence $\Q(\sqrt{\N(\al)})$ is also. As above, since $S_4$ has a unique subgroup
of order $12$, either this is equal to $\Q$, in which case $\al$ has
square norm, or it is equal to $\Q(\sqrt{D})$, in which case 
$\N(\al)=Da^2$ hence $\N(\al D)=D^4a^2$, so $\al D$ has square norm.

Conversely, given a noncyclic $k$ and $K_6 = k(\sqrt{\al})$ with $\al$ 
of square norm, one also adjoins $\sqrt{D}$, $\sqrt{\al D}$, and 
$\sqrt{\al'}$ for a conjugate $\al'$ of $\al$, and checks that the 
resulting field contains a square root of the remaining conjugate of $\al$ 
and is Galois over $\Q$ with Galois group $S_4$. One also checks that 
$\widetilde{K}$ only contains one quadratic extension of $K_6$, so that
$\Gal(\widetilde{K}/K_6) \simeq C_4$, and that there are four quartic 
subextensions $K$ of $\widetilde{K}$ which are all isomorphic.

\end{proof}

Since $K_6/k$ has trivial norm, there exists a positive integer $f$ such 
that $\N(\d(K_6/k))=f^2$, and we will write $f=f(K)$. Thus, if we denote by 
$\FF(k)$ the set of isomorphism classes of quartic extensions whose cubic 
resolvent is isomorphic to $k$ we have
$$\sum_{K\in\FF(k)}\dfrac{1}{|\Disc(K)|^s}=\dfrac{1}{|\Disc(k)|^s}\sum_{K\in\FF(k)}\dfrac{1}{f(K)^{2s}}\;.$$


The following proposition makes the correspondence of Theorem \ref{thm_quartic_corr} 
computationally explicit,
which helped us to check numerically the correctness of our formulas.

\begin{proposition}\label{propresal} \hfill
\begin{enumerate}[(1)]
\item A defining polynomial for the cubic resolvent field of the quartic 
field defined by the polynomial $x^4+a_3x^3+a_2x^2+a_1x+a_0$ is given by
$$x^3-a_2x^2+(a_1a_3-4a_0)x+4a_0a_2-a_1^2-a_0a_3^2\;,$$
whose (polynomial) discriminant is the same as the (polynomial) discriminant
of the quartic.
\item If $K_6=k(\sqrt{\al})$ with $\al\in k^*\setminus {k^*}^2$ of square norm
with characteristic polynomial $x^3+a_2x^2+a_1x+a_0$, a defining polynomial
for the corresponding quartic field is given by
$$x^4+2a_2x^2-8\sqrt{-a_0}x+a_2^2-4a_1\;,$$
whose (polynomial) discriminant is $2^{12}$ times the (polynomial) discriminant
of the cubic.
\end{enumerate}\end{proposition}

\begin{proof} (1). This is well-known: if $(\al_i)$ are the four roots of the 
quartic, the cubic is the characteristic polynomial of $\al_1\al_2+\al_3\al_4$.

\smallskip

(2). Assume that we are in the $S_4$ case, the $A_4$ case being simpler.
If we denote as usual by $\al'$ and $\al''$ the conjugates of $\al$,
then $\th=\sqrt{\al}$, $\th'=\sqrt{\al'}$, and $\th''=\sqrt{\al''}$ belong to 
$\widetilde{K}$, and we choose the square roots so that
$\th\th'\th''=\sqrt{\N(\al)}=\sqrt{-a_0}$. If we set $\eta=\th+\th'+\th''$,
by Galois theory it is clear that $\eta$ belongs to a quartic field, and
more precisely the four conjugates of $\eta$ are 
$\eps\th+\eps'\th'+\eps''\th''$ with the $\eps=\pm1$ such that 
$\eps\eps'\eps''=1$, and a small computation shows that the characteristic
polynomial of $\eta$ is the one given in the proposition. This polynomial
must be irreducible, because $\eta$ is fixed by a subgroup of $\Gal(\widetilde{K}/\Q)$ isomorphic
to $S_3$, but not by all of $\Gal(\widetilde{K}/\Q) \simeq S_4$, and there are
no intermediate subgroups of $S_4$ of order $12$.
\end{proof}



\section{The Main Theorem of \cite{Coh_a4s4}}\label{sec:CDO}

\subsection{Statement of the main theorem}

To prove our main result (Theorem \ref{thm_main}), we begin by stating and
proving a similar result involving sums over characters of certain ray class
groups instead of over quartic fields. This result has been stated without
proof in \cite{Coh_a4s4} and proved in the unpublished preprint
\cite{CDO_quartic}, so we give a complete proof here.

\begin{definition}\label{def_zkc}
For each ideal $\c\mid2\Z_k$ we define the following quantities:

\begin{enumerate}[(1)]
\item
We define a finite group $C_{\c^2}$ by\footnote{Note that there is a misprint in Definition 2.2 of \cite{Coh_a4s4}, the condition $\be\equiv1\pmod{^*\c^2}$ having been omitted from the denominator.}
$$C_{\c^2}=\dfrac{\{\a/\ (\a,\c)=1,\ \N(\a)\text{ square}\}}{\{\q^2\be/\ (\q^2\be,\c)=1,\ \be\equiv1\pmod{^*\c^2},\ \N(\be)\text{ square}\}},$$
and we define $X_{\c^2}$ to be the
group of characters 
$\chi\in C_{\c^2}$, extended to all ideals of square norm by setting
$\chi(\a)=0$ if $\a$ is not coprime to $\c$.
\item
We define $z_k(\mfc)$ to be equal to 1 or 2, with $z_k(\mfc) = 2$ if and only 
if we are in one of the following cases.
\begin{itemize}
\item We have $\mfc = 2 \Z_k$.
\item The prime $2$ splits as $2 \Z_k = \mfp_1^2 \mfp_2$ and 
$\mfc = \mfp_1 \mfp_2$.
\item The prime $2$ splits as $2 \Z_k = \mfp_1^2 \mfp_2$, $\mfc = \mfp_2$, 
and $\Disc(k) \equiv 4 \pmod 8.$
\item The prime $2$ splits as $2 \Z_k = \mfp_1^3$ and $\mfc = \mfp_1^2$.
\end{itemize}
\end{enumerate}
\end{definition}

Since trivially $C_{\c^2}$ has exponent dividing $2$, all the elements of 
$X_{\c^2}$ are quadratic characters, which can be applied only on ideals of 
square norm. With this definition, the main result of \cite{Coh_a4s4} is the 
following:

\begin{theorem}\label{thma4s4}\cite{Coh_a4s4, CDO_quartic}
$$\Phi_k(s)=\dfrac{2^{2-r_2(k)}}{a(k)2^{3s}}\sum_{\c\mid2\Z_k}\N\c^{s-1}z_k(\c)\prod_{\p\mid\c}\left(1-\dfrac{1}{\N\p^s}\right)\sum_{\chi\in X_{\c^2}}F_k(\chi,s)\;,$$
where $r_2(k)$ is half the number of complex places of $k$,
\begin{equation}\label{eqn_a4s4_euler}
F_k(\chi,s)=
\prod_{p\Z_k=\p_1\p_2}\left(1+\dfrac{\chi(\p_2)}{p^s}\right)\prod_{p\Z_k=\p_1^2\p_2}\left(1+\dfrac{\chi(\p_1\p_2)}{p^s}\right)
\prod_{p\Z_k=\p_1\p_2\p_3}\left(1+\dfrac{\chi(\p_1\p_2)+\chi(\p_1\p_3)+\chi(\p_2\p_3)}{p^s}\right)\;,\end{equation}
where in the product over $p\Z_k=\p_1\p_2$ it is understood that $\p_1$ has
degree $1$ and $\p_2$ has degree $2$.
\end{theorem}

\begin{remark} In \cite{CDO_quartic} this is proved for relative quartic extensions of any number field.
Proving Theorem \ref{thm_main} in this generality may be possible, but this
would involve additional technical complications.
Here we prove Theorem \ref{thma4s4} only for quartic extensions of $\Q$, allowing
some simplifications of the arguments in \cite{CDO_quartic}.
\end{remark}

Our first goal is the proof of this theorem, which may be summarized as follows.
By Theorem \ref{thm_quartic_corr}, it is enough to count quadratic extensions
$K_6/k$ of trivial norm. Our first few results parameterize such quadratic extensions
and allow us to compute their discriminants, and these together with some elementary computations
lead us to the preliminary result
of Corollary \ref{corphik}. To proceed further we use local class field theory to 
study a certain quantity $|S_{\c^2}[N]|/|C_{\c^2}|$
appearing in the corollary, allowing us to obtain the more explicit formula above. 

Later we will further refine this theorem and show that the groups $C_{\c^2}$ are isomorphic to certain ray
class groups. This will allow us to regard characters of $C_{\c^2}$ as
characters of Galois groups of quadratic extensions of $k$, allowing us to express the Euler products
above in terms of splitting of prime ideals in certain quartic fields.

\subsection{Proof of Theorem \ref{thma4s4}: Hecke, Galois, and Kummer Theory}\label{sec_CDO}

\begin{remark} Much of this section has been taken nearly verbatim from 
the unpublished preprint \cite{CDO_quartic}, and we thank the second and
third authors of that preprint for permission to include their results here.
\end{remark}
 
We begin by recalling the following (easy) special case of a theorem of Hecke
on Kummer extensions (see for example Theorem 10.2.9 of \cite{Coh}).

\begin{proposition}\label{prop:hec} Let $k$ be a number field and
$K_6=k(\sqrt{\al})$ be a quadratic extension of $k$. 
Write $\al\Z_k=\a\q^2$ where $\a$ is an integral squarefree ideal of $k$, and 
assume $\al$ is chosen so that $\q$ is coprime to $2$ (which can easily be done 
without changing the field $K_6$ by replacing $\q$ by $\ga\q$ for a suitable
element $\ga\in k^*$). 

Then the relative ideal discriminant of $K_6/k$ is 
given by the formula $\gd(K_6/k)=4\a/\c^2$, where $\c$ is the largest integral
ideal of $k$ (for divisibility) dividing $2\Z_k$, coprime to $\a$, and such 
that the congruence $x^2\equiv\al\pmod{^*\c^2}$ has a solution in $k$,
where the $^*$ has the usual multiplicative meaning of class field theory.
\end{proposition}

\begin{definition} Let $k$ be a number field.\begin{enumerate}
\item We say that an element $\al\in k^*$ is a $2$-virtual
unit if $\al\Z_k=\q^2$ for some ideal $\q$ of $k$, or equivalently, if
$v_{\p}(\al)\equiv0\pmod{2}$ for all prime ideals $\p$, and we denote by
$V(k)$ the group of virtual units.
\item We define $V[N]$ as the subgroup of elements of $V(k)$ having
square (or equivalently, positive) norm.
\item We define $V^+(k)$ as the subgroup of totally positive 
virtual units, so that $V^+(k)\subset V[N]$.
\item We define the $2$-Selmer group $S(k)$ of $k$ as 
$S(k)=V(k)/{k^*}^2$, and similarly $S[N]=V[N]/{k^*}^2\subset S(k)$,
and $S^+(k)=V^+(k)/{k^*}^2\subset S[N]$.
\end{enumerate}\end{definition}

\begin{remarks} \hfill
\begin{itemize}
\item We should more properly speaking write $V_2(k)$, $S_2(k)$, etc...,
but in this paper we only consider $2$-virtual units.
\item If $k$ is a cubic field (or more generally a field of odd degree), which
will be the case in this paper, and $\al\in V(k)$, then either $\al$ or
$-\al$ belongs to $V[N]$.
\item If $k$ is a complex cubic field, we have $V^+(k)=V[N]$ and
$S^+(k)=S[N]$, since here totally positive means positive for the unique
real embedding.\end{itemize}\end{remarks}

\begin{lemma}\label{lem_square} The relative discriminant $\gd(K_6/k)$ of a 
quadratic extension $K_6/k$ divides $(4)$ if and only if $K_6=k(\sqrt{\al})$ 
for some $\al\in V(k)$. If this is the case then $\gd(K_6/k)$ is a square,
and $K_6/k$ is unramified at infinity if and only if $\al\in V^+(k)$.
\end{lemma}

\begin{proof} If $\al\in V(k)$ then $\al\Z_k=\q^2$, where as before we may choose $\q$ coprime to $2$,
so that $\gd(K_6/k)=4/\c^2$ by
Hecke. Conversely, let $K_6=k(\sqrt{\al})$, and write $\al\Z_k=\a\q^2$, again
with $\q$ coprime to $2$. Then $\gd(K_6/k)=4\a/\c^2$
by Hecke, with $\c$ coprime to $\a$, so that $\gd(K_6/k)\mid (4)$
if and only if $\a=\Z_k$, i.e., $\al\in V(k)$. For the last
statement, note that $K_6/k$ is unramified at infinity if and only if
$\al$ is totally positive.
\end{proof}

Note that by definition, if $L\in\LL_2(k)$ then $L$ is totally real if $k$ 
is so, hence if $K_6=k(\sqrt{\al})$ we have $\al\in V^+(k)$.

\smallskip

The classification of quadratic extensions of trivial norm (see Definition
\ref{deftriv}) is easily done as follows. 

\begin{proposition}\label{prop:kum} There is a one-to-one 
correspondence between on the one hand quadratic extensions of $k$ of trivial 
norm, together with the trivial extension $k/k$, and on the other hand pairs 
$(\a,\ov{u})$, where $\a$ is an integral, squarefree ideal of $k$ of square 
norm whose class modulo principal ideals is a square in the class group of 
$k$, and $\ov{u}\in S[N]$.
\end{proposition}

\begin{proof} The exact sequence
$$1\LR S(k)\LR\dfrac{k^*}{{k^*}^2}\LR\dfrac{I(k)}{I(k)^2}\LR\dfrac{\Cl(k)}{\Cl(k)^2}\LR1\;,$$
where as usual $I(k)$ is the group of nonzero fractional ideals of $k$, shows
the trivial fact that there is a one-to-one correspondence between quadratic
extensions (including $k/k$) and pairs $(\a,\overline{u})$ with $\a$
integral and squarefree whose class belongs to $\Cl^2(k)$ and 
$\overline{u}\in S(k)$, and the trivial norm condition is equivalent to the
restrictions on $(\a,\overline{u})$.

We can make the correspondence explicit as follows.
For each ideal $\a$ as above, choose arbitrarily an
ideal $\q_0=\q_0(\a)$ and an element $\al_0=\al_0(\a)$ such that 
$\a\q_0^2=\al_0\Z_k$. Since $k$ is a cubic field, by changing $\al_0$ into
$-\al_0$ if necessary we may assume that $\al_0$ has square norm.
Thus if $K_6=k(\sqrt{\al})$ with $\al$ of square norm and $\al\Z_k=\a\q^2$,
with $\a$ integral and squarefree, then $\al/\al_0(\a)\in V[N]$ and the 
corresponding pair is $(\a,\ov{\al/\al_0(\a)})$.
Conversely, for any pair $(\a,\ov{u})$ as above we take 
$K_6=k(\sqrt{\al_0(\a)u})$ for any lift $u$ of $\ov{u}$.
\end{proof}


We now begin the computation of $\Phi_k(s)$. Recall that to any $A_4$ or 
$S_4$-quartic field there correspond $a(k)$ extensions $K_6/k$ of trivial 
norm, where $a(k)=3$ in the $A_4$ case and $a(k)=1$ in the $S_4$ case,
so by definition
$$\Phi_k(2s)=\dfrac{1}{a(k)}+\sum_{K\in\FF(k)}\dfrac{1}{f(K)^{2s}}
=\dfrac{1}{a(k)}\sum_{K_6/k\text{ of trivial norm}}\dfrac{1}{\N(\d(K_6/k))^s}\;,$$
where it is understood that we include the trivial extension $k/k$. Thus,
if $(\a,\ov{u})$ is as in Proposition \ref{prop:kum} and
$K_6=k(\sqrt{\al_0(\a)u})$ is the corresponding quadratic extension, we know
from Proposition \ref{prop:hec} that $\gd(K_6/k)=4\a/\c^2$ for an
ideal $\c=\c(\a,\ov{u})$ described in the proposition. For ease of notation,
we let $\calA$ be the set of all ideals $\a$ as in Proposition \ref{prop:kum}
above, in other words, integral, squarefree ideals of $k$ of square norm 
whose class in $\Cl(k)$ belongs to $\Cl(k)^2$. 
We thus have
$$a(k)\Phi_k(2s)=\sum_{\a\in\calA}\sum_{\ov{u}\in S[N]}\dfrac{1}{\NO(\gd(k(\sqrt{\al_0(\a)u})/k))^s}=\dfrac{1}{4^{3s}}\sum_{\a\in\calA}\dfrac{1}{\N\a^s}\sum_{\ov{u}\in S[N]}\NO(\c(\a,\ov{u}))^{2s}\;,$$
in other words
$$a(k)\Phi_k(s)=\dfrac{1}{2^{3s}}\sum_{\a\in\calA}\dfrac{1}{\N\a^{s/2}}S(\a;s)\text{\quad with\quad}S(\a;s)=\sum_{\ov{u}\in S[N]}\NO(\c(\a,\ov{u}))^s=\sum_{\substack{\c\mid2\Z_k\\(\c,\a)=1}}\N\c^sT(\a,\c)\;,$$
where 
$T(\a,\c)=\left|\left\{\ov{u}\in S[N]/\ \c(\a,\ov{u})=\c\right\}\right|$.

\begin{definition} We set
$$f(\a,\c)=\left|\left\{\ov{u}\in S[N]/\ x^2\equiv\al_0u\pmod{^*\c^2}\text{ soluble}\right\}\right|\;.$$
\end{definition}

\begin{lemma}\label{lem_prelim} We have the preliminary formula
\begin{align*}a(k)\Phi_k(s) = \dfrac{1}{2^{3s}}\sum_{\c\mid2\Z_k}\N\c^s\prod_{\p\mid\c}\left(1-\N\p^{-s}\right)\sum_{\substack{\a\in\calA\\(\c,\a)=1}}\dfrac{1}{\N\a^{s/2}}f(\a,\c)\;.\end{align*}
\end{lemma}

\begin{proof} By definition of $\c(\a,\ov{u})$, for every $\c\mid2\Z_k$ we have
$$f(\a,\c)=\sum_{\substack{\c\mid\c_1\mid2\Z_k\\(\c_1,\a)=1}}T(\a,\c_1)\;.$$
By dual M\"obius inversion
we obtain
$$T(\a,\c)=\sum_{\substack{\c | \c_1 | 2 \Z_k \\(\c_1,\a)=1}}\mu_k\Big(\frac{\c_1}{\c}\Big)f(\a,\c_1)\;,$$
where $\mu_k$ is the M\"obius function on ideals of $k$. Replacing in the 
formula for $S(\a;s)$, we thus have

\begin{align*}S(\a;s)&=\sum_{\substack{\c\mid2\Z_k\\(\c,\a)=1}}\N\c^sT(\a,\c)
=\sum_{\substack{\c_1\mid2\Z_k\\(\c_1,\a)=1}}f(\a,\c_1)\sum_{\c\mid\c_1}\mu_k(\c_1/\c)\N\c^s\\
&=\sum_{\substack{\c_1\mid2\Z_k\\(\c_1,\a)=1}}f(\a,\c_1)\N\c_1^s\prod_{\p\mid\c_1}\left(1-\N\p^{-s}\right)\;.\end{align*}

Replacing in the formula for $\Phi_k(s)$ and writing $\c$ for $\c'$, we obtain (when $(\c, \a) = 1$)

\begin{align*}a(k)\Phi_k(s)&=\dfrac{1}{2^{3s}}\sum_{\a\in{\calA}}\dfrac{1}{\N\a^{s/2}}\sum_{\substack{\c\mid2\Z_k\\(\c,\a)=1}}f(\a,\c)\N\c^s\prod_{\p\mid\c}\left(1-\N\p^{-s}\right)\\
&=\dfrac{1}{2^{3s}}\sum_{\c\mid2\Z_k}\N\c^s\prod_{\p\mid\c}\left(1-\N\p^{-s}\right)
\sum_{\substack{\a\in{\calA}\\(\c,\a)=1}}\dfrac{1}{\N\a^{s/2}}f(\a,\c)\;,\end{align*}
proving the lemma.\end{proof}

To compute $f(\mfa, \mfc)$ we introduce the following definitions.

\begin{definition}\label{defsa4} Let $k$ be a cubic field and let $\c$ be an 
ideal of $k$ dividing $2\Z_k$.
\begin{enumerate}\item We define the square ray class group modulo $\c^2$ by
$$\Cl_{\c^2}[N]=\dfrac{\{\a/\ (\a,\c)=1,\ \N(\a)\text{ square}\}}
{\{\be\Z_k/\ \be\equiv1\pmod{^*\c^2},\ \N(\be)\text{ square}\}}\;.$$
In particular, we set $\Cl[N]=\Cl_{\Z_k}[N]$.
\item We define the following subgroup of $\Cl_{\c^2}[N]$:
$$D_{\c^2}[N]=\dfrac{\{\q^2\be/\ (\q^2\be,\c)=1,\ \be\equiv1\pmod{^*\c^2},\ \N(\be)\text{ square}\}}
{\{\be\Z_k/\ \be\equiv1\pmod{^*\c^2},\ \N(\be)\text{ square}\}}\;.$$
Note that the group $C_{\c^2}$, defined in Definition \ref{def_zkc}, is
canonically isomorphic to $\Cl_{\c^2}[N] / D_{\c^2}[N]$.
\item We define the ordinary ray Selmer group modulo $\c^2$ by
$$S_{\c^2}(k)=\{\ov{u}\in S(k)/\ x^2\equiv u\pmod{^*\c^2}\quad\text{soluble}\}\;.$$
\item We define the square Selmer group modulo $\c^2$ by
$$S_{\c^2}[N]=\{\ov{u}\in S_{\c^2}(k)/\ \N(u)\text{ square}\}\;.$$
(Observe that our Selmer group definitions do not depend on the choice of lifts $\ov{u}$.)
\item\label{it_def_zkc_real} Set $Z_{\c}=(\Z_k/\c^2)^*$, and let $Z_{\c}[N]$ 
be the subgroup of elements of $Z_{\c}$ which have a lift to $\Z_k$ whose 
norm is a square. We define $z_k(\c)$ as the index of $Z_{\c}[N]$ in $Z_{\c}$.
(In Proposition \ref{prop_compute_zkc} we will prove that
this intrinsic definition of $z_k(\c)$ agrees with Definition \ref{def_zkc}.)
\end{enumerate}
\end{definition}

\begin{remark} 
 Note that if $u\in V(k)$ then $\pm\N(u)$ is a square
for a suitable sign, so in the definition above the condition $\N(u)$ square
can be replaced by $\N(u)>0$. (This simplification did not apply in \cite{CDO_quartic},
as the base field there was not necessarily $\Q$.)
\end{remark}

The value of $f(\a,\c)$ is then given by the following proposition.

\begin{proposition}\label{prop:acns} Let $\a\in\calA$ with 
$(\a,\c)=1$ be as above. We have $f(\a,\c)\neq0$ if and only if the class of 
$\a$ in $\Cl_{\c^2}[N]$ belongs in fact to $D_{\c^2}[N]$, in which case
$f(\a,\c)=|S_{\c^2}[N]|$.\end{proposition}

\begin{proof} This is just a matter of rewriting the definitions. Assume that
there exists $\ov{u}\in S[N]$ such that $x^2\equiv\al_0u\pmod{^*\c^2}$ has a 
solution. This means that we can write $\be x^2=\al_0u$ for some
$\be\equiv1\pmod{^*\c^2}$. Since $u\in V(k)$ we can write 
$u\Z_k=\q_1^2$ for some ideal $\q_1$. Thus 
$$\a=\al_0\q_0^{-2}=(\be x^2/u)\q_0^{-2}=\be(x/(\q_0\q_1))^2\;.$$
In addition, since $u$ and $\al_0$ are of square norm in $K$, $\be$ is also of
square norm, and since $\a$ is coprime to $\c$, the class of $\a$ in 
$\Cl_{\c^2}[N]$ belongs to $D_{\c^2}[N]$. The proof of the converse
retraces the above steps and is left to the reader, proving the first part of 
the proposition. For the second part, assume that there exists 
$\ov{v}\in S[N]$ and $y\in k^*$ such that $y^2\equiv\al_0v\pmod{^*\c^2}$. Then
the solubility of $x^2\equiv\al_0u\pmod{^*\c^2}$ is equivalent to that of
$(x/y)^2\equiv(u/v)\pmod{^*\c^2}$, in other words to 
$\ov{u}\in\ov{v}\,S_{\c^2}[N]$ whose cardinality is equal to that of 
$S_{\c^2}[N]$, proving the proposition.
\end{proof}

\begin{corollary}\label{corphik} We have
\begin{align*}a(k)\Phi_k(s)&=\dfrac{1}{2^{3s}}\sum_{\c\mid2\Z_k}\dfrac{\left|S_{\c^2}[N]\right|}{|C_{\c^2}|}\N\c^s\prod_{\p\mid\c}\left(1-\N\p^{-s}\right)\sum_{\chi\in X_{\c^2}}F_k(\chi,s)\;,\text{\quad with}\\
F_k(\chi,s)&=\prod_{p\Z_k=\p_1\p_2}\left(1+\dfrac{\chi(\p_2)}{p^s}\right)
\prod_{p\Z_k=\p_1^2\p_2}\left(1+\dfrac{\chi(\p_1\p_2)}{p^s}\right)\prod_{p\Z_k=\p_1\p_2\p_3}\left(1+\dfrac{\chi(\p_1\p_2)+\chi(\p_1\p_3)+\chi(\p_2\p_3)}{p^s}\right)\;.\end{align*}
\end{corollary}

\begin{proof} By the above proposition we have
$$a(k)\Phi_k(s)=\dfrac{1}{2^{3s}}\sum_{\c\mid2\Z_k}\left|S_{\c^2}[N]\right|\N\c^s\prod_{\p\mid\c}\left(1-\N\p^{-s}\right)\sum_{\substack{\a\in\calA\\\ov{\a}\in D_{\c^2}[N]}}\dfrac{1}{\N\a^{s/2}}\;.$$
Note first that if $\ov{\a}\in D_{\c^2}[N]$ then the class of $\a$ in $\Cl(k)$
belongs to $\Cl(k)^2$, so we may replace $\a\in\calA$ by $\a\in\calA'$, where
the condition that the class of $\a$ is in $\Cl(k)^2$ is removed. Since 
$C_{\c^2} = \Cl_{\c^2}[N]/D_{\c^2}[N]$, we can detect the 
condition $\ov{\a}\in D_{\c^2}[N]$ by summing over characters of $C_{\c^2}$, 
hence
$$a(k)\Phi_k(s)=\dfrac{1}{2^{3s}}\sum_{\c\mid2\Z_k}\dfrac{\left|S_{\c^2}[N]\right|}{|C_{\c^2}|}\N\c^s\prod_{\p\mid\c}\left(1-\N\p^{-s}\right)\sum_{\chi\in X_{\c^2}}\sum_{\a\in\calA'}\dfrac{\chi(\a)}{\N\a^{s/2}}\;.$$
The last sum is clearly multiplicative (because we removed the condition on
$\Cl(k)^2$), and looking at the five possible decomposition types of primes
and keeping only the integral squarefree ideals of square norm proves the 
corollary.\end{proof}

\subsection{Computation of $|S_{\c^2}[N]|/|C_{\c^2}|$}

The above computations were straightforward. It remains to
compute $|S_{\c^2}[N]|/|C_{\c^2}|$, and this is more difficult.
In the course of this computation we will prove some intermediate results
which we will further use in this paper.

\begin{proposition}\label{prop:exseq} Recall from 
Definition \ref{defsa4} that we have set $Z_{\c}=(\Z_k/\c^2)^*$, and that
$C_{\c^2}=\Cl_{\c^2}[N]/D_{\c^2}[N]$.
\begin{enumerate}
\item There exists a natural exact sequence
\begin{equation}\label{eqn:exseq}
1\LR S_{\c^2}[N]\LR S[N]\LR \dfrac{Z_{\c}[N]}{Z_{\c}^2}\LR C_{\c^2}\LR C_{(1)}
\LR1\;.
\end{equation}
\item There exists a natural exact sequence
\begin{equation}
1 \LR U(k)/U(k)^2 \LR S(k) \LR \Cl(k)[2] \LR 1\;,
\end{equation}
and a canonical isomorphism $S[N]\isom S(k)/\{\pm1\}$.
\item We have a canonical isomorphism $C_{(1)}\isom\Cl(k)/\Cl(k)^2$.
\end{enumerate}
\end{proposition}

\begin{proof} (1). All the maps are clear, and the exactness is immediate 
everywhere except for the surjectivity of the final map. Let $\a$ be an ideal 
of square norm, so that $\N(\a)=q^2$ with $q \in \Q$. If we set
$\b=\a q^{-1}$ it is clear that $\b/\N(\b)=\a$. By the approximation
theorem, we can find $\be\in k$ such that $\be\b$ is integral and coprime
to $2\Z_k$. It follows that $\N(\be\b)$ is coprime to $2$, hence
$(\be/\N(\be))\b/\N(\b)=(\be/\N(\be))\a$ is coprime to
$2\Z_k$, hence to $\c^2$. Since $\be/\N(\be)$ is an element of
square norm, we conclude that $(\be/\N(\be))\a$ is in the same class
as $\a$ in $\Cl[N]$ and is coprime to $\c^2$, proving the surjectivity of
the natural map from $\Cl_{\c^2}[N]$ to $\Cl[N]$, hence that of the last
map in the sequence above.

\smallskip

(2) and (3). The exact sequence is equivalent to the definition of $S(k)$, and
the isomorphism $S(k)/\{\pm1\}\isom S[N]$ is simply the map induced by
$u\mapsto\sign(\N(u))u$. Finally (3) will be proved in Proposition
\ref{prop_c4} below.\end{proof}

\begin{corollary}\label{corsc2} We have
$$\dfrac{|S_{\c^2}[N]|}{|C_{\c^2}|}=2^{2-r_2(k)}\dfrac{z_k(\c)}{\N(\c)}\;,$$
where we recall from Definition \ref{defsa4} that $z_k(\c)=[Z_{\c}:Z_{\c}[N]]$.
\end{corollary}

\begin{proof} Using the exact sequences and isomorphisms, the equality 
$|\Cl(k)[2]|=|\Cl(k)/\Cl(k)^2|$ and Dirichlet's unit theorem telling us that 
$|U(k)/U(k)^2|=2^{r_1(k)+r_2(k)}$ we have
\begin{align*}|S_{\c^2}[N]||Z_{\c}[N]/Z_{\c^2}||C_{(1)}|&=|S[N]||C_{\c^2}|=|S(k)||C_{\c^2}|/2=|U(k)/U(k)^2||\Cl(k)[2]||C_{\c^2}|/2\\
&=2^{r_1(k)+r_2(k)}|\Cl(k)/\Cl(k)^2||C_{\c^2}|/2=2^{r_1(k)+r_2(k)-1}|C_{(1)}||C_{\c^2}|\;,\end{align*}
in other words 
$$|S_{\c^2}[N]|/|C_{\c^2}|=2^{r_1(k)+r_2(k)-1}/|Z_{\c}[N]/Z_{\c}^2|\;.$$
Since $k$ is a cubic field we have $r_1(k)+r_2(k)=3-r_2(k)$, so
the corollary is immediate from the following lemma:\end{proof}

\begin{lemma}\label{easyzk1} The map $x\mapsto x^2$ induces a natural
isomorphism between $(\Z_k/\c)^*$ and $Z_{\c}^2$. In particular, 
$|Z_{\c}^2|=\phi_k(\c)$, where $\phi_k$ denotes the ideal Euler $\phi$ function
for the field $k$, $|Z_{\c}/Z_{\c}^2|=\N(\c)$ and 
$|Z_{\c}[N]/Z_{\c}^2|=\N(\c)/z_k(\c)$.
\end{lemma}

\begin{proof} Since $\c\mid2\Z_k$, it is immediately checked that the 
congruence $x^2\equiv1\pmod{^*\c^2}$ is equivalent to $x\equiv1\pmod{^*\c}$, 
which proves the first result. Now it is well-known and easy to show that we
have the explicit formula $\phi_k(\c)=\N(\c)\prod_{\p\mid\c}(1-1/\N\p)$. It
follows in particular that 
$|Z_{\c}|:=|\Z_k/\c^2|=\phi_k(\c^2)=\N(\c)\phi_k(\c)$, so that
$|Z_{\c}/Z_{\c}^2|=\N(\c)=|Z_{\c}/Z_{\c}[N]||Z_{\c}[N]/Z_{\c}^2|$,
proving the lemma by definition of $z_k(\c)$.
\end{proof}

It remains to prove the formula for $z_k(\c)$ given in Definition 
\ref{def_zkc}.

\subsection{Computation of $z_k(\c)$}

We first prove a series of lemmas.

\begin{lemma}\label{propm1} 
For each $\mfc | 2 \Z_k$ we have $z_k(\c)=1$ if there
exists an element $\be$ of square norm such that $\be\equiv-1\pmod{\c^2}$,
and $z_k(\c)=2$ otherwise.

In particular, $z_k((2)) = 2$.
\end{lemma}

\begin{proof}
Let $\ov{\al}\in(\Z_k/\c^2)^*$, and choose a lift of $\ov{\al}$ to 
$\Z_k$ which is coprime to $2$, which is always possible. Then 
$\N(\al)$ is odd hence $\N(\al)\equiv\pm1\pmod{4}$. If $\N(\al)\equiv1\pmod4$,
then $\al/\N(\al)\equiv\al\pmod{4}$ and a fortiori modulo $\c^2$, and
is of square norm. 

If $\be$ as described in the lemma exists, 
and $\N(\al)\equiv-1\pmod4$, then 
$\al\be/\N(\al)\equiv\al\pmod{\c^2}$ and is of square norm, so that $z_k(\c) = 1$.
If no such $\beta$ exists, then $-1$ 
lacks a lift of square norm so that $z_k(\c) \geq 2$. However,
for each $\overline{\alpha}$,
one of $\overline{\alpha}$ or 
$- \overline{\alpha}$ has a lift of square norm: again choose a lift $\alpha$ of $\overline{\alpha}$
coprime to $2$, and either $\N(\al) \equiv 1\pmod4$ and $\alpha/\N(\alpha) \equiv \alpha \pmod 4$ is
of square norm, or this is true with $\alpha$ replaced by $-\alpha$. Therefore $z_k(\c) = 2$.

For $\c = (2)$, the condition $\beta \equiv -1 \ (\textmod \ 4)$
implies $\N(\beta) \equiv -1 \ (\textmod \ 4)$, so that $\N(\beta)$ cannot be a
square. 
\end{proof}

\begin{lemma}\label{lem_sn} An element $\beta \in k$ is of square norm if and only if $\beta = \alpha/\N(\alpha)$
for some $\alpha \in k$.
\end{lemma}
\begin{proof} Given $\alpha$, we see immediately that $\beta$ is of square norm; conversely, given $\beta$,
take $\alpha = \beta/\sqrt{\N(\beta)}$. 
\end{proof}
Of course this is also equivalent to the condition that $\beta = \N(\alpha')/\alpha'$ for some $\alpha' \in k$.

\begin{lemma}\label{unramc} If $\p$ is an unramified prime ideal dividing $2$ 
then $z_k(2/\p)=1$.\end{lemma}

\begin{proof} Set $\c=2/\p$. Since $\p$ is unramified we have $\p\nmid\c$, 
so the inclusions $k \hookrightarrow k_{\p}$ and $k \hookrightarrow k_{\c}$
induce an isomorphism $k\otimes\Q_2\isom k_{\c}\times k_{\p}$; 
here
$k_{\p}$ is the completion of $k$ at $\p$ and $k_{\c}$ is isomorphic to the
product of the completions of $k$ at primes other than $\p$. Any element $\ga$
of $k\otimes\Q_2$ can thus be written in the form $(\ga_{\c},\ga_{\p})$ and
we have 
$$\N(\ga)=\N_{k_{\c}/\Q_2}(\ga_c)\N_{k_{\p}/\Q_2}(\ga_{\p})\;.$$
It follows that $\N(\ga)/\ga=(\al_{\c},\al_{\p})$ with
$$\al_{\c}=\dfrac{\N_{k_{\c}/\Q_2}(\ga_c)}{\ga_c}\N_{k_{\p}/\Q_2}(\ga_{\p})\;.$$
We choose $\ga_{\c}=1$. Furthermore, since $\p$ is unramified, we know that
the local norm from $k_{\p}$ to $\Q_2$ is surjective on units, so in particular
there exists $\ga_{\p}\in k_{\p}$ such that $\N_{k_{\p}/\Q_2}(\ga_{\p})=-1$.
It follows that for such a $\ga$, we have 
$\N(\ga)/\ga=(-1,u)$ for some $u\in k_{\p}$, and in particular the
local component at $\c$ of $\N(\ga)/\ga$ is equal to $-1$. By density (or,
equivalently, by the approximation theorem), we can find $\ga$ in the global
field $k$ such that $\N(\ga)/\ga\equiv-1\pmod{\c^m}$ for any $m\ge1$, and
in particular for $m=2$. If we set $\be=\N(\ga)/\ga$, it is clear that
$\be$ is of square norm and $\be\equiv-1\pmod{\c^2}$. We conclude thanks to
Lemma \ref{propm1}.\end{proof}

\begin{lemma}\label{codiff} Let $\al\in k$ be such that $\al/2\in\Z_k$ and
$\al/4\in\Gd^{-1}(k)$, where $\Gd^{-1}(k)$ denotes the codifferent of $k$.
Then the characteristic polynomial of $\al$ is congruent to $X^3$ modulo $4$,
and in particular the norm of $\al-1$ is not a square in $\Q$.\end{lemma}

\begin{proof} Let $X^3-tX^2+sX-n$ be the characteristic polynomial of $\al$.
Since $\al/4\in\Gd^{-1}(k)$, we know that $t=\Tr(\al)\equiv0\pmod4$.
Furthermore, since $\al/2\in\Z_k$, we also have $\Tr(\al(\al/2))\equiv0\pmod4$,
hence $\Tr(\al^2)\equiv0\pmod8$, so $s=(\Tr(\al)^2-\Tr(\al^2))/2\equiv0\pmod4$.
Finally we have $n=\N(\al)\equiv0\pmod8$, and in particular $n\equiv0\pmod4$,
proving the first statement (note that in fact we can easily prove that 
$n\equiv0\pmod{16}$, but we do not need this). It follows that
the characteristic polynomial of $\al-1$ is congruent to $(X+1)^3$ modulo $4$,
hence that the norm of $\al-1$ is congruent to $-1$ modulo $4$, so cannot
be a square in $\Q$.\end{proof}

We can now finally compute $z_k(\c)$ in all cases:

\begin{theorem}\label{prop_compute_zkc} The definition of $z_k(\c)$ in
Proposition \ref{defsa4} \eqref{it_def_zkc_real} matches the explicit formula
of Definition \ref{def_zkc}.
\end{theorem}

\begin{proof} We have trivially $z_k((1))=1$, and by Lemma \ref{propm1} we
have $z_k((2))=2$, so we assume that $\c\ne(1)$ and $\c\ne(2)$. Lemma \ref{propm1}
also implies that $z_k(\c') \leq z_k(\c)$ for $\c' | \c$.

We consider
the possible splitting types of $2$ in $k$.\begin{itemize}
\item Assume that $2$ is unramified, and let $\p$ be a prime
ideal dividing $2/\c$. By Lemma \ref{unramc} we have $z_k(2/\p)=1$,
and since $\c\mid 2/\p$ we have $z_k(\c)\mid z_k(2/\p)$ so
$z_k(\c)=1$.
\item Assume that $2$ is totally ramified as $2\Z_k=\p^3$. If $\c=\p$
then $\be=1$ is of square norm and is such that $\be\equiv-1\pmod{\c^2}$, hence
$z_k(\p)=1$. If $\c=\p^2$, let $\be$ be any element such that 
$\be\equiv-1\pmod{\c^2}$, and set $\al=\be+1\in2\p$. Since $2$ is tamely 
ramified, we have $\Gd(k)=\p^2\a$ for some ideal $\a$ coprime to $2$. It 
follows that $\al/4\in\p^{-2}\subset\Gd^{-1}(k)$, and of course 
$\al/2\in\p\subset\Z_k$. Applying Lemma \ref{codiff}, we deduce that $\N(\be)$
cannot be a square, showing that $z_k(\p^2)=2$.
\item Assume that $2$ is partially ramified as $2\Z_k=\p_1^2\p_2$
and that $\p_1\mid\c$. Lemma \ref{unramc} implies that $z_k(\p_1^2)=1$, hence 
also $z_k(\p_1)=1$. If $\c=\p_1\p_2$, let $\be\equiv-1\pmod{\c^2}$, and set 
$\al=\be+1\in2\p_2$. Here $2$ is wildly ramified, so we also deduce
that $\Gd(k)=\p_1^2\a$ for some $\a$ not necessarily coprime to $2$.
As above, we deduce that 
$\al/4\in\p_1^{-2}\subset\Gd^{-1}(k)$ and $\al/2\in\p_2\subset\Z_k$, so that 
by Lemma \ref{codiff}, $\N(\be)$ cannot be a square, hence $z_k(\p_1\p_2)=2$.
\item Assume finally that $2$ is partially ramified as $2\Z_k=\p_1^2\p_2$
and that $\p_1\nmid\c$, in other words $\c=\p_2$. 
We use Lemma \ref{lem_sn} and the same local reasoning as for the
proof of Lemma \ref{unramc}. Setting $k_i=k_{\p_i}$, we can write 
$k\otimes\Q_2\isom k_1\times k_2$. If $\ga=(\ga_1,\ga_2)$
then as before we have $\N(\ga)/\ga=(\al_1,\al_2)$ with
$\al_2=N_{k_1/\Q_2}(\ga_1)$, 
since we can identify $k_2$ with $\Q_2$.
Since $2$ is a local uniformizer for the unramified prime ideal $\p_2$,
it follows that we have $z_k(\p_2)=1$ if and only if we can solve
$\al_2\equiv-1\pmod4$, hence if and only if there exists a norm congruent
to $-1$ modulo $4$ in the local ramified extension $k_1/\Q_2$. Up to
isomorphism, there are $6$ possible such extensions $k_1=\Q_2(\sqrt{D})$
with $D=-4$, $-8$, $-24$, $8$, $12$, and $24$ (the extension corresponding
to $D=-3$ is of course unramified). By inspection\footnote{For example,
$1^2 + 6 \cdot 1^2 = 7 \equiv -1 \pmod 8$, and as the conductor of $\Q_2(\sqrt{-6}) = \Q_2(\sqrt{-24})$
is $2^3$, this congruence can be lifted to a solution in $\Q_2$.}, we see that $-1$ is
a norm in the extensions corresponding to $D=-24$ and $D=8$, and that $3$
is a norm in the extensions corresponding to $D=-8$ and $D=24$. As $D\equiv 0 \ (\textmod \ 8)$ if and only if 
$\Disc(k) \equiv 0 \ (\textmod \ 8)$, this finishes the proof of the theorem 
in the case $\Disc(k) \equiv 0 \ (\textmod \ 8)$.

In case $\Disc(k) \equiv 4 \ (\textmod \ 8)$, we have $D = -4$ or $D = 12$,
and since the equations $x^2+y^2=-1$ and $x^2-3y^2=-1$ are not soluble in
$\Q_2$ (they do not have solutions modulo $4$) it follows that $-1$ is not a 
norm, so that $z_k(\p_2)=2$.
\end{itemize}
\end{proof}

Replacing the value of $z_k(\c)$ given by Definition \ref{def_zkc} in
the formula for $|S_{\c^2}[N]|/|C_{\c^2}|$ given in Corollary \ref{corsc2}
and then in the formula for $\Phi_k(s)$ given by Corollary \ref{corphik}
finishes the proof of Theorem \ref{thma4s4}.

\section{Study of the Groups $C_{\c^2}$}\label{sec_c}

We denote by $\rk_2(k)$ the $2$-rank of $\Cl(k)$, and by $\rk^+_2(k)$ the
$2$-rank of the narrow class group $\Cl^+(k)$.

\begin{proposition}\label{prop_c4}
Let $C_{(4)}$ be defined as in Definition \ref{def_zkc}.

\begin{enumerate}[1.]
\item\label{it_iso}
The map $\phi: \mfa \rightarrow \mfa / \N(\mfa)$ induces isomorphisms
\begin{equation}\label{eqn_1_iso}
\Cl(k) / \Cl(k)^2
\xrightarrow{\phi}
C_{(1)}\;,
\end{equation}
\begin{equation}\label{eqn_4_iso}
\Cl_{(4)}(k) / \Cl_{(4)}(k)^2
\xrightarrow{\phi}
C_{(4)}\;,
\end{equation}
with the inverse map induced by $\phi^{-1}: \mfb \rightarrow \mfb / \sqrt{N(\mfb)}$.

\item\label{it_size}
We have either $|C_{(4)}| = 2 |C_{(1)}|$ or $|C_{(4)}| = |C_{(1)}|$. Moreover, $|C_{(4)}| = |C_{(1)}|$ if and only if
$k$ is totally real and $\rk_2^+(k) = \rk_2(k)$, i.e., if and only if $k$ is
totally real and there does not exist a nonsquare totally positive unit.
\item\label{it_all_iso}
If $\c' | \c$ and $|C_{\c^2}| = |C_{(1)}|$ then $C_{\c'^2} \isom C_{(1)}$ also; and if
$\c' | \c$ and $|C_{\c'^2}| = |C_{(4)}|$ then $C_{\c'^2} \isom C_{(4)}$ also.

In particular, if $|C_{(1)}| = |C_{(4)}|$ then $C_{\mfc^2} \isom \Cl(k) / \Cl(k)^2 \isom 
\Cl_{(4)}(k) / \Cl_{(4)}(k)^2$ for every $\mfc | (2)$.
\end{enumerate}
\end{proposition}

\begin{proof}\vfill

\eqref{it_iso}. It is easily checked that $\phi^{-1}$ and $\phi$ yield
inverse bijections between the group of fractional ideals of $k$ of square 
norm and all fractional ideals of $k$. Both $\phi$ and $\phi^{-1}$ map 
principal ideals to principal ideals, and $\phi$ maps any square ideal to a
square ideal. Also it is clear that $\al \equiv 1 \ (\textmod \ 4\Z_k)$ 
implies $\N(\al)\equiv1\pmod4$,
so the maps \eqref{eqn_1_iso} and \eqref{eqn_4_iso} are well-defined homomorphisms. Conversely, it is
clear that the map $\phi^{-1}$ from $C_{(1)}$ to $\Cl(k)/\Cl(k)^2$ is also
well-defined. Consider $\phi^{-1}$ on $C_{(4)}$:  we have
$\phi^{-1}(\q^2\be)=\q^2\be/\N(\q)\sqrt{\N(\be)}$, and since 
$\be\equiv1\pmod{^*4\Z_k}$, the class of $\phi^{-1}(\q^2\be)$ in
$\Cl_{(4)}(k)/\Cl_{(4)}(k)^2$ is the same as that of $a\Z_k$ with
$a=\N(\q)\sqrt{\N(\be)}$. However $a\in\Q$, so $a\equiv\pm1\pmod{^*4}$,
and $a\Z_k=-a\Z_k$, so the class of $\phi^{-1}(\q^2\be)$ is trivial,
proving that $\phi^{-1}$ is a well-defined map from $C_{(4)}$ to
$\Cl_{(4)}(k)/\Cl_{(4)}(k)^2$, proving \eqref{it_iso}.
\\
\\
\eqref{it_size}. By Lemmas \ref{easyzk1} and \ref{propm1} we have
$|Z_{(2)}[N]/Z_{(2)}^2|=8/2=4$, so by the exact sequence of Proposition 
\ref{prop:exseq} $|C_{(4)}|$ is equal to $|C_{(1)}|$, $2 |C_{(1)}|$, 
or $4|C_{(1)}|$. However, it cannot be $4|C_{(1)}|$: if it were, by the same
exact sequence this would imply that $S_{(4)}[N] = S[N]$ , so
that if any $u \in k$, coprime to $2$, represents an element of $S[N]$, the
equation $x^2 \equiv u \ (\textmod \ {}^* 4)$
is soluble. By Proposition \ref{prop:hec} this implies that each element of 
$S[N]$ yields a quadratic extension of $k$ of discriminant $(1)$;
however, if $|C_{(4)}| > |C_{(1)}|$, then there exist quadratic extensions 
of $k$ of nontrivial square discriminant dividing $(4)$ and unramified at 
infinity, hence by Lemma \ref{lem_square} generated by virtual units of square
norm, a contradiction.

For the second statement, note that the $1^s$ coefficient of 
$a(k)2^{r_2(k)} \Phi_k(s)$ is equal to $|C_{(4)}|$. 
This may be seen from Theorem \ref{thma4s4}, or equivalently, it is a 
consequence of various computations in this section.
This $1^s$ coefficient counts the number of quadratic or trivial extensions 
$K_6 =k(\sqrt{\al})$, where $\al$ is of square norm, and $K_6/k$
is unramified at all finite places. 

By \eqref{it_iso}, $|C_{(1)}|$ counts the number of such extensions $K_6/k$ 
which are also unramified at infinity. Thus, in the totally real case, 
$|C_{(4)}| > |C_{(1)}|$ if and only if $\rk_2^+(k) \neq \rk_2(k)$. In the
complex case, due to the factor of $2^{r_2(k)} = 2$, we necessarily
have $|C_{(4)}| > |C_{(1)}|$. 

By the ray class group exact sequence, we have $\rk^+_2(k)=\rk_2(k)$ if and 
only if $\Cl^+(k)=\Cl(k)$ if and only if $[U(k):U^+(k)]=2^{r_1}=[U(k):U^2(k)]$,
and since $U^2(k)\subset U^+(k)$, if and only if $U^+(k)=U^2(k)$, proving our
claim.
\\
\\
\eqref{it_all_iso}. This follows from the existence of natural surjections
$C_{\mfc} \LR C_{\mfc'}$ for $\mfc' | \mfc$.
\end{proof}

The following is independent of the rest of this paper but is an immediate
consequence of the above proof. It seems that there should exist a simple direct proof of this corollary, which would 
provide a simpler proof of \eqref{it_size} above, but we did not find such a proof.

\begin{corollary} Let $k$ be a cubic field. There exists a virtual unit $u$
coprime to $2$ and of square norm such that $u\not\equiv1\pmod{4\Z_k}$.
\end{corollary}

\begin{proof} Indeed, since the class of a virtual unit in $S(k)$ can always
be represented by a virtual unit coprime to $2$, the statement is equivalent
to the statement that the natural injection from $S_{(4)}[N]$ to $S[N]$
is not surjective, and by the exact sequence of Proposition
\ref{prop:exseq}, that $|C_{(4)}|\ne4|C_{(1)}|$, proved above.
\end{proof}

In most cases, Proposition \ref{prop_c4} will suffice to handle the groups $C_{\c^2}$, but in two special cases
where $C_{\c^2} \simeq C_{(4)}$, we will need to evaluate $\chi(\mfa)$ for characters $\chi \in X_{\c^2}$
on ideals $\mfa$ which are coprime to $\c^2$ but not $4$. For this 
we require the following refinement. 

\begin{proposition}\label{prop_c4_ext}
Suppose that $(2) = \c \c'$ with $\c$ and $\c'$ coprime and squarefree and $\N(\c') = 4$; namely, either
\begin{itemize}
\item $(2) = \p_1 \p_2$ with each $\p_i$ of degree $i$, $\c = \p_1$, and $\c' = \p_2$, or
\item $(2) = \p_1 \p_2 \p_3$ with each $\p_i$ of degree $1$, $\c = \p_1$, and $\c' = \p_2 \p_3$.
\end{itemize}
Then, the map $\mfa \rightarrow \mfa / \N(\mfa)$ induces an isomorphism 
\begin{equation}\label{eqn_4c_iso}
\Cl_{\c'^2}(k) / \Cl_{\c'^2}(k)^2
\xrightarrow{\phi}
C_{\c^2}\;,
\end{equation}
which agrees with \eqref{eqn_4_iso} on ideals coprime to $(2)$, and for which $\phi(\p_1^2) = 4/\p_1^2$.
\end{proposition}

\begin{proof}
Because $2$ is unramified and $\N(\c) = 2$, the map $\mfa \rightarrow \mfa/\N(\mfa)$ sends ideals coprime
to $\c'$ to ideals coprime to $\c$.
The map $\mfb \rightarrow \mfb/\sqrt{\N(\mfb)}$ does {\itshape not} necessarily
send ideals coprime to $\c$ to ideals coprime to $\c'$, for example if $(2) = \p_1 \p_2 \p_3$, $\c = \p_1$, and
$\b = \p_2^2$. 

As before we argue that $\phi$ is well defined:
If $\alpha \equiv 1 \pmod{^* \c'^2}$, then $\beta := \alpha / \N(\alpha)$ must be coprime to $2$
(the $\mfp_1$-adic valuations of $\alpha$ and $\N(\alpha)$ must be equal), and we must have $\pm \beta \equiv 1
\pmod{^* \p_1^2}$, so that $(\beta)$ represents the trivial class of $C_{\c^2}$.

It seems difficult to work with $\phi^{-1}$ directly, so 
we apply the fact that {\itshape both groups in \eqref{eqn_4c_iso} are of the same size}\footnote{This fact is
proved in
Proposition \ref{prop_list_c}, which does not in turn rely on the proposition being proved. We might have
waited until Section \ref{sec_quartic_proofs} to give the current proof, but we placed it here for readability's sake.}, 
and prove that $\phi$ is surjective. By \eqref{eqn_4c_iso}, we know that these
groups are either the same size as $\Cl(k)/\Cl(k)^2$ and $C_{(1)}$ respectively, in which case they
are canonically isomorphic to them and the result is immediate, or else these groups are twice as large
as $\Cl(k)/\Cl(k)^2$ and $C_{(1)}$.

In the latter case choose $\alpha \equiv 3 \pmod{^* \p_1^2}$ coprime to $2$, and if needed
(as in Section \ref{sec_sig})
take $\alpha$ to be totally positive by adding a large multiple of $4$.
Then
$\beta := \alpha / \N(\alpha)$ is a totally positive element of square norm with $\beta \equiv 3 \pmod{^* \p_1^2}$,
with $(\beta)$ in the image of \eqref{eqn_4c_iso}, and
which represents a nontrivial class of $C_{\p_1^2}$ but the trivial class of $C_{(1)}$. 
Since $C_{\p_1^2}$ is twice as large as $C_{(1)}$, we conclude that \eqref{eqn_4c_iso} is surjective.
\end{proof}

\section{Splitting Types in $k$, $K_6$, and $L$}\label{sec_xnew}

By the results of the previous section, together with class field theory, we can translate characters
of $C_{\c^2}$ into characters of Galois groups of quadratic extensions. It is therefore important to understand
the possible ways in which primes can split in $k$, $K_6$, and $L$.
The following gives a complete answer to this question:

\begin{theorem}\label{thm_splitting_types} Let $L$ be an $A_4$ or 
$S_4$-quartic field, and let $k$ be its cubic resolvent.

Suppose first that $p \geq 3$ is a prime number.
\begin{enumerate}[1.]
\item If $p$ is $(3)$ in $k$, then $p$ is $(3 1)$ in $L$.
\item If $p$ is $(21)$ in $k$, then $p$ is $(4)$, $(2 11)$, $(2^2)$, or 
$(1^2 1^2)$ in $L$.
\item If $p$ is $(111)$ in $k$, then $p$ is $(1111), (22), (2^2),$ or 
$(1^2 1^2)$ in $L$.
\item If $p$ is $(1^2 1)$ in $k$, then $p$ is $(2 1^2)$, $(1^2 11)$, or 
$(1^4)$ in $L$.
\item If $p$ is $(1^3)$ in $k$, then $p$ is $(1^3 1)$ in $L$.
\end{enumerate}
If $p = 2$, then in addition to the above decomposition types, in all cases 
$2$ can be $(1^4)$ in $L$, if $2$ is $(1^2 1)$ in $k$ then $2$ can also be
$(1^2 1^2)$ in $L$, and if $2$ is $(1^2 1)_4$ in $k$ then $2$ can also be
$(2^2)$ in $L$.

Moreover, we have the Artin relation 
\begin{equation}\label{eqn_artin}
\zeta_L(s) = \frac{ \zeta(s) \zeta_{K_6}(s)}{ \zeta_k(s) }
\end{equation}
among the Dedekind zeta functions associated to $L$, $K_6$, $k$, and $\Q$, allowing us to determine
the splitting types of a prime $p$ in any of $k$, $L$, and $K_6$ from the splitting 
in the remaining two fields.

\end{theorem}

We remark that Theorem \ref{thm_splitting_types} overlaps substantially with unpublished work of Martinet \cite{Mar}.

\begin{proof} 
The equation \eqref{eqn_artin} is a well-known consequence of the character theory of $A_4$ and $S_4$,
and we omit the details here.

For the first part of the theorem, the basic idea of the proof is as follows: A prime $p$ may have splitting type $(3)$, $(21)$, $(111)$, 
$(1^2 1)$, or $(1^3)$ in $k$, and each of the primes above $p$ may be ramified, split, or inert in $K_6$;
by \eqref{eqn_artin}, this determines the splitting type of $p$ in $L$. 

By computer search, we found triples $(k, K_6, L)$ for each combination of splitting types listed in the theorem.
It therefore remains to prove that no other combinations are possible. Several tools useful for this include:

\begin{itemize} 

\item
We can use \eqref{eqn_artin} to rule out some
combinations, for example, $(21)$ in $k$ and $(411)$ in $K_6$.

\item A theorem of Stickelberger says that if $F$ is a number field 
of degree $n$ and $p$ is a prime unramified in $F$ which splits into $g$ prime 
ideals, then $\leg{\Disc(F)}{p}=(-1)^{n-g}$. Thus, since 
$\Disc(L)=\Disc(k) f(L)^2$, it follows that
when $p$ is unramified both in $k$ and $L$ the number of primes above $p$ in 
$k$ and $L$ must have opposite parity. This rules out, for example, the possibility that $p$
is
$(21)$ in $k$ and $(222)$ in $K_6$.
\item Using the square norm condition: recall that $\gd(K_6/k)=4\a/\c^2$
with $\a$ integral, squarefree, and of square norm. Thus if $\p$ is a prime 
ideal of $k$ not dividing $2$ which is ramified in $K_6/k$, we must have
$v_{\p}(\a)=1$, and a short computation shows that
$\sum_{\p\mid p\Z_k, \p\mid\gd(K_6/k)}f(\p/p)$ is even. This rules out, 
for example, the possibility that a prime $p \neq 2$ splits as $(21)$ in $k$ and $(221^2)$
in $K_6$. 

This condition does not apply to $p = 2$, and indeed more splitting types are possible for $p = 2$.

\item Using divisibility by $3$ of ramification degrees: suppose that a prime
$p$ splits as $(1^31)$ in $L/\Q$. If $\mfP$ is a prime of $\widetilde{L}$
above $p$, the ramification index $e(\mfP | p)$ must be divisible by $3$, implying that
$p$ is $(1^3)$ in $k$.
\end{itemize}
This leaves nine additional cases to rule out, which split further into subcases depending on whether $L$ is an $S_4$ or $A_4$-quartic field.
We accomplished this using a variety of
group-theoretic arguments. Although the proofs are entertaining, they would require a somewhat lengthy digression.
Accordingly we omit them here, but an unpublished note with complete proofs is available from the second author's website\footnote{
\url{http://www.math.sc.edu/~thornef}}. Here, we give
a single example illustrating the flavor of these proofs.
\begin{itemize}
\item
Suppose that $p$ is $(21)$ in $k$
$(221^2)$ in $K_6$, and $(21^2)$ in $L$, where $L$ is an $S_4$-quartic field.
Writing $[\widetilde{L} : \Q] = efg$ with the usual meaning, 
we have $2\mid e$ and $2\mid f$. 

If
$\GP$ is an ideal of $\WL$ above the ideal $\p$ of $L$ with 
$f(\p|p)=2$, we have $e(\GP|\p)\mid[\WL:L]=6$, and since $e(\p|p)=1$
it follows that $e=e(\GP|p)\mid 6$, so that $e=2$. Similarly, by considering 
the ideal $\p$ of $L$ with $e(\p|p)=2$ we see that $f=2$. Thus the 
decomposition fields are quartic fields, and since the only quartic subfields
of $\WL$ are the conjugates of $L$, it follows that $L$ is a 
decomposition field, a contradiction since none of the prime ideals $\p$
of $L$ above $p$ satisfy $e(\p|p)=f(\p|p)=1$.
\end{itemize}
\end{proof}

We will also need the following simple consequence of \eqref{eqn_artin} in the sequel.
\begin{proposition}\label{revised_cortot}
A prime $p$ is $(1^4)$ in $L$ if and only if all the prime 
ideals above $p$ in $k$ are ramified in the quadratic extension $K_6/k$.
\end{proposition}

\section{The arithmetic of quartic fields in $\LL_2(k)$}\label{sec_arith_quartic}
Recall that in Definition \ref{defll} we wrote
$$\LL_2(k)=\LL(k,1)\cup\LL(k,4)\cup\LL(k,16)\cup\LL_{tr}(k,64)\;.$$

The point of this definition is the following crucial result:

\begin{theorem}\label{thmcrux} We have $L\in\LL_2(k)$ if and only if the 
corresponding extension of trivial norm as explained above is of the form 
$K_6=k(\sqrt{\be})$, where $\be\in V^+(k)$ is coprime to $2$.
\end{theorem}

\begin{proof}
Assume first that $\be\in V^+(k)$. By Proposition
\ref{prop:hec} and Lemma \ref{lem_square} we have $\d(K_6/k)=4/\c^2$, so
$\N(\d(K_6/k))=\Disc(L)/\Disc(k)=64/\N(\c)^2$. Since $\be$ is totally positive,
when $k$ is totally real so are $K_6$ and the Galois closure of $L$, so
$L$ is totally real. If $\c\ne\Z_k$ we have $64/\N(\c)^2=1$, $4$, or $16$,
so $L\in\LL_2(k)$. Thus assume that $\c=\Z_k$, so that $\Disc(L)=64\Disc(k)$
and $\d(K_6/k)=4$. This implies that all the primes above $2$ in $k$ are
ramified in $K_6/k$, so by Proposition \ref{revised_cortot} the prime $2$ is totally 
ramified in $L$.

\smallskip

Conversely, let $L\in\LL_2(k)$ and let $K_6=k(\sqrt{\al})$ be the 
corresponding extension, where $\al$ is of square norm. Write 
$\al\Z_k=\a\q^2$, where $\a$ is unique if we choose it integral and squarefree,
and $\q$ can be chosen coprime to $2$. Note that if $k$
is totally real then so is $L$ and hence $K_6$, so $\al$ will automatically
be totally positive. Since $\al$ has square
norm, so does $\a$. Since $\d(K_6/k)=4\a/\c^2$ with $\c\mid2\Z_k$ coprime to
$\a$ and $\N(\d(K_6/k))=\Disc(L)/\Disc(k)=2^{2j}$ for $0\le j\le 3$, it 
follows that $\a$ is a product of distinct prime ideals above $2$, whose 
product of norms is a square. If $\a=\Z_k$ then $\be = \al$ is a virtual unit
coprime to $2$.
Thus, assume by contradiction that $\a\ne\Z_k$. 
Considering the five possible splitting types of $2$ in $k$ and using the
fact that $\N(\a)\le\N(\c^2)$, it is immediate to see that the only remaining
possibilities are as follows:

\begin{itemize}\item $2\Z_k=\p_1\p_2$ with $\p_2$ of degree $2$, 
$\a=\p_2$, $\c=\p_1$, $\d(K_6/k)=\p_2^3$, hence $L\in\LL(k,64)$.
Then $\p_1$ is not ramified in $K_6/k$ so by Proposition \ref{revised_cortot}
$2$ is not totally ramified in $L$, a contradiction.
\item $2\Z_k=\p_1\p_2\p_3$, $\a=\p_1\p_2$, $\c=\p_3$, $\d(K_6/k)=(\p_1\p_2)^3$,
hence $L\in\LL(k,64)$. Similarly, here $\p_3$ does not ramify in $K_6/k$, so 
$2$ is not totally ramified in $L$, again a contradiction and proving the
theorem.\end{itemize}
\end{proof}

In light of Lemma \ref{lem_square}, we immediately obtain the following:

\begin{corollary}\label{corcrux}
Suppose that $L$ is an $A_4$ or $S_4$-quartic field, with $K_6$ the 
corresponding quadratic extension of $k$. Then $L \in \calL_2(k)$
if and only if $K_6/k$ is unramified at infinity and $\mfd(K_6/k) | (4)$.
\end{corollary}

The analysis for $A_4$-quartic fields is greatly simplified by the following 
result.
\begin{proposition}\label{prop_53}
If $k$ is a cyclic cubic field then $\rk_2(k)$ is even and 
$\rk_2^+(k) = \rk_2(k)$.

In addition, if $a \in \Z$, the $2$-ranks of $\Cl_{a \Z_k}(k)$ and of
$\Cl^+_{a \Z_k}(k)$ are also even.
\end{proposition}
\begin{proof}
There is a natural action of the group ring $\Z[G]$ on $\Cl(k)$ and on 
$\Cl^+(k)$, and in the cyclic cubic case $G=\langle\sigma\rangle$, where
$1+\sigma+\sigma^2$ acts trivially, so we have an action of
$\Z[\sigma]/(1+\sigma+\sigma^2)\isom\Z[\z_3]$. Since $2$ is inert in 
$\Z[\z_3]$ it follows that any $2$-torsion $\Z[\z_3]$-module has even rank,
so in particular $\rk_2(k)$ and $\rk^+_2(k)$ are even. Furthermore, by
a theorem of Armitage and Fr\"ohlich \cite{AF}, for any number field $K$ with
$r_1$ real embeddings we have 
\begin{equation}\label{eqn_af}
\rk^+_2(K)\le \rk_2(K)+\lfloor r_1/2\rfloor\;,
\end{equation}
so in our case $\rk^+_2(k)\le \rk_2(k)+1$, so that we have equality since
both ranks are even.

The final statement is true for the same reason: the $G$-invariance of
$a\Z_k$ guarantees that $\Cl_{a\Z_k}(k)$ and $\Cl^+_{a\Z_k}(k)$ are also
$\Z[\z_3]$-modules.
\end{proof}

These results, combined with our previous work (especially Proposition 
\ref{prop_c4}), imply the following counting formulas:

\begin{proposition}\label{prop_l2k}
The following statements are true:
\begin{enumerate}[(1.)]
\item\label{it_l2k_count} We have 
$|\calL(k, 1)| = (2^{\rk_2(k)} - 1)/a(k) = (|C_{(1)}| - 1)/a(k)$,
where $a(k)$ is as in Definition \ref{defphik}.
\item\label{it_l2k_count2}
We have $|\calL_2(k)| = |C_{(4)}| - 1$, so that for noncyclic $k$, $|\calL_2(k)|$ is equal to either $|\calL(k, 1)|$ or $2 |\calL(k, 1)| + 1$.
\item\label{it_l2k_empty}
We have $\calL(k, 4) = \calL(k, 16) = \calL_{tr}(k, 64) = \emptyset$ {\upshape (}equivalently, $|\calL_2(k)| = |\calL(k, 1)|$
{\upshape )}
if and only if $k$ is totally real and $\rk_2^+(k) = \rk_2(k)$, i.e., if and 
only $k$ is totally real and there does not exist a nonsquare totally positive
unit. In particular, this is true for cyclic fields.
\item\label{it_l2k_only_one}
If one of $\calL(k, 4)$, $\calL(k, 16)$, or $\calL_{tr}(k, 64)$ is nonempty 
then the other two are empty.
\end{enumerate}
\end{proposition}

\begin{proof}

\eqref{it_l2k_count}. By Theorem \ref{thm_quartic_corr}, the elements of
$\calL(k, 1)$ are in $a(k)$-to-$1$ correspondence with the 
quadratic extensions of $k$ which are unramified everywhere (including at the
infinite places). By class field theory, they correspond to subgroups
of $\Cl(k)$ of index $2$, yielding the first equality. The second equality is
a consequence of \eqref{it_iso} of Proposition \ref{prop_c4}.

\eqref{it_l2k_count2}. The elements of $\calL_2(k)$ may have the ramification 
described in Corollary \ref{corcrux}, and the equality again follows by class 
field theory and Proposition \ref{prop_c4}.

\eqref{it_l2k_empty}. This follows from \eqref{it_size} of Proposition 
\ref{prop_c4}: the latter criterion is equivalent to the equality 
$|C_{(4)}| = |C_{(1)}|$. The statement for cyclic fields follows from 
Proposition \ref{prop_53}.

\eqref{it_l2k_only_one}
If $|C_{(4)}| \neq |C_{(1)}|$, then $|\Cl_4(k)[2]| = 2|\Cl(k)[2]|$. Among the
ideals dividing $(4)$, choose a minimal ideal $\mfm$ with the property that
$|\Cl_{\mfm}(k)[2]| = 2|\Cl(k)[2]|$. Then $\mfm$ must be a square by Lemma 
\ref{lem_square}, and all extensions $K_6/k$ counted by $\Cl_{(4)}(k)[2]$ must
have $\mfd(K_6/k) = \Z_k$ or $\mfd(K_6/k) = \mfm$. Therefore,
all of the fields in $\calL_2(k)$ have discriminant equal to either $\Disc(k)$
or $\Disc(k) \cdot \NO(\mfm)$, with $\NO(\mfm)$ equal to $4$, $16$, or $64$.
\end{proof}

In Section \ref{sec_arith_s4} we prove further results about the arithmetic of
$S_4$-quartic fields in particular. We conclude this section with the 
following result, which is an analogue for $A_4$-quartic fields of a related result in \cite{CT3}.
As with that proposition the result is not required 
elsewhere in this paper; rather, it is an application of Theorem
\ref{thm_main} (2).

\begin{proposition} Let $k$ be a cyclic cubic field such that $\rk_2(k) = 4$,
so that there exist five $A_4$-quartic fields $L$ with cubic resolvent $k$, 
which all satisfy $\Disc(L) = \Disc(k)$. If $2$ is totally split in $k$, then 
$2$ is totally split in exactly one of the five $A_4$-quartic fields $L$,
and splits as $2 \Z_L = \mfp_1 \mfp_2$ with $\mfp_i$ of degree $2$ in the four
others.
\end{proposition}

\begin{proof}
Writing $D = \Disc(k)$, we use Theorem \ref{thm_main} (2) to count the number 
of cubic fields of discriminant $D$, $16D$, $64D$, and $256D$. By hypothesis 
(and Proposition \ref{prop_53}), $\rk_2(\Cl(k)) = \rk_2(\Cl^+(k)) = 4$.

Write $5 = a + b$, where of the five quartic fields $L$,
$2$ is totally split in $a$ of them, and $2$ is $(22)$ in $b$ of them.
By Theorem \ref{thm_splitting_types} these are no other possibilities, and we prove that
$a = 1$.

The $1^{-s}, 2^{-s}, 4^{-s}, 8^{-s},$ and $16^{-s}$ coefficients of $\Phi_k(s)$ are equal to
$\frac{16}{3}, 0, 16, 8a - 8, 8a - 8$ respectively. Therefore, the number of quadratic extensions
$K_6 / k$, ramified only at infinity and/or 2, 
is equal to $3 \big(- \frac{1}{3} + \frac{16}{3} + 16 + 2 \cdot(8a - 8) \big) = 15 + 48a$.
These extensions, together with the trivial extension $k/k$, are counted by the ray class group
$\Cl^+_{(8)}(k) / \Cl^+_{(8)}(k)^2$. By Proposition \ref{prop_53}, this has even $2$-rank. Hence,
$16 + 48a = 2^r$ where $r \geq 6$ is even. (We know that $r \neq 4$ because the $4^{-s}$ coefficient is nonzero.)

In addition to $a = 1$ and $r = 6$, this equation also has one other solution $a = 5$ and $r = 8$, and we conclude the proof
by ruling this out. Consider the ray class group exact sequence \cite[Proposition 3.2.3]{Coh}
\begin{equation}
1 \rightarrow U^+_{(n)}(k) \rightarrow U(k) \rightarrow (\Z_k / n\Z_k)^{\times} \times \mathbb{F}_2^3 \rightarrow
\Cl^+_{(n)}(k) \rightarrow \Cl(k) \rightarrow 1\;,
\end{equation}
for $n = 1$ and $n = 8$. (Here $U^+_{(n)}(k)$ is the group of totally positive units congruent to $1 \ (\textmod \ {}^{*}n)$.)
As $\Cl^+(k)/\Cl^+(k)^2 \simeq \Cl(k)/\Cl(k)^2$, we know that the image of $U(k)$ surjects onto $\mathbb{F}_2^3$, i.e., that
all eight sign signatures are represented by units of $k$. Since $(\Z_k/8\Z_k)^{\times}$ has $2$-rank $3$ (recall that $2$ is totally split), 
this implies that
$\rk_2 ( \Cl^+_{(8)}(k) ) \leq \rk_2 ( \Cl(k) ) + 3 = 4 + 3 = 7$. Therefore $r \neq 8$ and the proof is complete.
\end{proof}
\begin{remark}
It ought to be possible to prove similar statements for related situations (e.g. if $\rk_2(k) = 6$), but we have not pursued this.
\end{remark}

\section{The arithmetic of $S_4$-quartic fields in $\LL_2(k)$}\label{sec_arith_s4}

In this section we further study the set $\LL_2(k)$ in the (more complicated)
$S_4$ case. If $k$ is an $S_3$-cubic field, then by Proposition \ref{prop_l2k},
at most one of $\calL(k, 4)$, $\calL(k, 16)$, and $\calL_{tr}(k, 64)$ can be
nonempty. We will prove the following:

\begin{proposition}\label{prop_split_table}
Let $k$ be an $S_3$-cubic field and let $L \in \LL_2(k)$. Then the following
table gives a complete list of all possibilities for the following data:
\begin{itemize}
\item The splitting type of $2$ in $k$, $K_6$, and $L$, and if it is $(1^2 1)$
in $k$ we also include as an index $\Disc(k) \ (\textmod \ 8)$;
\item The quantity $n^2:=\Disc(L)/\Disc(k)$, which must be equal to
$1$, $4$, $16$, or $64$.
\end{itemize}
For each possible combination we give a cubic polynomial generating $k$ over 
$\Q$, the characteristic polynomial $P_{\al}(x)$ of $\al$ of square norm such 
that $K_6=k(\sqrt{\al})$, and a defining polynomial for $L$ over~$\Q$.
\end{proposition}

\bigskip

\centerline{
\begin{tabular}{|c|c|c|c||c|c|c|}
\hline
$k$-split & $K_6$-split & $L$-split & $n^2$ & $k$ & $P_{\al}(x)$ & $L$ \\
\hline\hline
$(3)$& $(33)$ & $(31)$ & $1$ & $x^3 - x^2 - 14x + 23$ & $x^3 - 35x^2 + 179x - 81$ & $x^4 - x^3 - 4x^2 + x + 2$ \\
\hline
$(3)$ & $(3^2)$ & $(1^4)$ & $64$ & $x^3 - x^2 - 4x + 3$ & $x^3 - 15x^2 + 55x - 49$ & $x^4 - 2x^3 - 6x^2 + 2$ \\
\hline
$(21)$& $(42)$ & $(4)$ & $1$ & $x^3 - x^2 - 9x + 10$ & $x^3 - 8x^2 + 12x - 1$ & $x^4 - 4x^2 - x + 1$ \\
\hline
$(21)$& $(2211)$ & $(211)$ & $1$ & $x^3 - 20x - 17$ & $x^3 - 12x^2 + 28x - 1$ & $x^4 - 6x^2 - x + 2$ \\
\hline
$(21)$& $(2^22)$ & $(2^2)$ & $16$ & $x^3 - x^2 - 14x - 4$ & $x^3 - 22x^2 + 21x - 4$ & $x^4 - 11x^2 - 2x + 25$ \\
\hline
$(21)$& $(2^211)$ & $(1^2 1^2)$ & $16$ & $x^3 - x^2 - 7x + 6$ & $x^3 - 13x^2 + 36x - 16$ & $x^4 - 2x^3 - 5x^2 + 2x + 2$ \\
\hline
$(21)$& $(2^21^2)$ & $(1^4)$ & $64$ & $x^3 - 4x - 1$ & $x^3 - 11x^2 + 19x - 1$ & $x^4 - 2x^3 - 4x^2 + 4x + 2$ \\
\hline
$(111)$& $(2211)$ & $(22)$ & $1$ & $x^3 - x^2 - 34x - 16$ & $x^3 - 67x^2 + 947x - 625$ & $x^4 - x^3 - 8x^2 + x + 3$ \\
\hline
$(111)$& $(21^21^2)$ & $(2^2)$ & $16$ & $x^3 - 13x- 4$ & $x^3 - 14x^2 + 45x - 4$ & $x^4 - 7x^2 - 2x + 1$ \\
\hline
$(111)$& $(111111)$ & $(1111)$ & $1$ & $x^3 + x^2 - 148x + 480$ & $x^3 - 37x^2 + 308x - 576$ & $x^4 - 2x^3 - 17x^2 - 6x + 16$ \\
\hline
$(111)$& $(1^21^211)$ & $(1^2 1^2)$ & $16$ & $x^3 - 17x - 8$ & $x^3 - 26x^2 + 65x - 36$ & $x^4 - 13x^2 - 6x + 26$ \\
\hline
$(111)$& $(1^21^21^2)$ & $(1^4)$ & --- & --- & --- & --- \\
\hline
\end{tabular}}
\smallskip
\centerline{
\begin{tabular}{|c|c|c|c||c|c|c|}
\hline
$(1^21)_0$& $(2^211)$ & $(21^2)$ & $1$ & $x^3 - x^2 - 18x - 14$ & $x^3 - 43x^2 + 323x - 25$ & $x^4 - x^3 - 5x^2 + 2x + 2$ \\
\hline
$(1^21)_0$& $(1^21^211)$ & $(1^211)$ & $1$ & $x^3 - x^2 - 58x + 186$ & $x^3 - 75x^2 + 499x - 169$ & $x^4 - x^3 - 9x^2 + 3x + 14$ \\
\hline
$(1^21)_0$& $(1^411)$ & $(1^2 1^2)$ & $4$ & $x^3 - 22x - 8$ & $x^3 - 14x^2 + 25x - 4$ & $x^4 - 7x^2 - 2x + 6$ \\
\hline
$(1^21)_0$& $(1^41^2)$ & $(1^4)$ & $64$ & $x^3 - x^2 - 6x + 2$ & $x^3 - 16x^2 + 60x - 16$ & $x^4 - 8x^2 - 4x + 1$ \\
\hline
$(1^21)_4$& $(2^211)$ & $(21^2)$ & $1$ & $x^3 - x^2 - 20x - 22$ & $x^3 - 43x^2 + 291x - 121$ & $x^4 - x^3 - 5x^2 + 4x + 2$ \\
\hline
$(1^21)_4$& $(1^21^211)$ & $(1^211)$ & $1$ & $x^3 - 59x - 168$ & $x^3 - 91x^2 + 915x - 1849$ & $x^4 - x^3 - 11x^2 + 11x + 16$ \\
\hline
$(1^21)_4$& $(1^42)$ & $(2^2)$ & $4$ & $x^3 - 11x - 12$ & $x^3 - 10x^2 + 21x - 4$ & $x^4 - 5x^2 - 2x + 1$ \\
\hline
$(1^21)_4$& $(1^42)$ & $(2^2)$ & $16$ & $x^3 - x^2 - 8x + 10$ & $x^3 - 13x^2 + 32x - 16$ & $x^4 - 2x^3 - 5x^2 + 2x + 3$ \\
\hline
$(1^21)_4$& $(1^411)$ & $(1^2 1^2)$ & $16$ & $x^3 -22x - 20$ & $x^3 - 17x^2 + 64x - 16$ & $x^4 - 2x^3 - 7x^2 + 4x + 2$ \\
\hline
$(1^21)_4$& $(1^41^2)$ & $(1^4)$ & --- & --- & ---  & --- \\
\hline
$(1^3)$& $(1^31^3)$ & $(1^3 1)$ & $1$ & $x^3 - x^2 - 27x - 43$ & $x^3 - 43x^2 + 179x - 9$ & $x^4 - x^3 - 5x^2 + 3x + 4$ \\
\hline
$(1^3)$& $(1^6)$ & $(1^4)$ & $4$ & $x^3 - x^2 - 9x + 11$ & $x^3 - 12x^2 + 16x - 4$ & $x^4 - 6x^2 - 2x + 5$ \\
\hline
$(1^3)$& $(1^6)$ & $(1^4)$ & $64$ & $x^3 - x^2 - 7x - 3$ & $x^3 - 20x^2 + 104x - 144$ & $x^4 - 10x^2 - 12x - 1$ \\
\hline
\end{tabular}}

\bigskip

\begin{remarks}\begin{itemize}
\item The empty rows in the table correspond to combinations not 
ruled out for $p = 2$ in the table in Theorem \ref{thm_splitting_types}, 
but which we will prove cannot occur for $L \in \LL_2(k)$.
\item In each case we have chosen the noncyclic totally real cubic field $k$ 
with smallest discriminant satisfying the given splitting conditions for $k$, 
$L$, and value of $n^2=\Disc(L)/\Disc(k)$, and in addition such that
$|\LL_2(k)|\le1$. One could also in most cases also choose complex cubic fields
if desired.
\item We could simplify this table by giving in addition to the splittings,
only the $P_{\al}(x)$ column, since $P_{\al}$ also generates the field $k$,
the field $K_6$ is given by the polynomial $P_{\al}(x^2)$, and the field
$L$ by Proposition \ref{propresal}. We have preferred to keep it as above.
\item As with Theorem \ref{thm_splitting_types}, our results in this section partially overlap with unpublished work
of Martinet \cite{Mar}.
\end{itemize}
\end{remarks}
 
In particular, the above proposition and table implies that for the $S_4$ case,
the table of Theorem \ref{thm_main} covers precisely the cases that occur. For
the $A_4$ case this is easy to check using Theorem
\ref{thm_splitting_types} alone.

Most of Proposition \ref{prop_split_table} is contained within the following:

\begin{proposition}\label{propdist} Let $k$ be a cubic field such that
either $k$ is complex, or $k$ is totally real and $\rk^+_2(k)>\rk_2(k)$, so 
that by Proposition \ref{prop_l2k}, one and exactly one of $\LL(k,4)$, 
$\LL(k,16)$, or $\LL_{tr}(k,64)$ is nonempty. Denote temporarily by
$W^+(k)$ the group of $u\in V^+(k)$ coprime to $2$, and by
$W_{(4)}^+(k)$ the group of $u\in W^+(k)$ such that $x^2\equiv u\pmod{^*(4)}$
is soluble (equivalently, $\ov{u}\in S^+_{(4)}(k)$).
\begin{enumerate}

\item\label{it_dist_in} If $2$ is $(3)$ in $k$, then 
$\LL_{tr}(k,64)\ne\emptyset$.
\item\label{it_dist_ps} If $2$ is $(21)$ in $k$, write $2\Z_k=\p_1\p_2$ with
$\p_i$ of degree $i$. Then if $v_{\p_1}(\al-1)=1$ for some $\al\in W^+(k)$
we have $\LL(k,16)=\emptyset$ and $\LL_{tr}(k,64)\ne\emptyset$, while if
$v_{\p_1}(\al-1)\ge2$ for some $\al\in W^+(k)\setminus W_{(4)}^+(k)$ the
reverse is true.
\item\label{it_dist_ts} If $2$ is $(111)$ in $k$ then $\LL(k,16)\ne\emptyset$.
\item\label{it_dist_pr0} If $2$ is $(1^21)_0$ in $k$, write 
$2\Z_k=\p_1^2\p_2$. Then if $v_{\p_1}(\al-1)=1$ for some $\al\in W^+(k)$ we 
have $\LL(k,4)=\emptyset$ and $\LL_{tr}(k,64)\ne\emptyset$, while if
$v_{\p_1}(\al-1)\ge2$ for some $\al\in W^+(k)\setminus W_{(4)}^+(k)$ the
reverse is true.
\item\label{it_dist_pr4} If $2$ is $(1^21)_4$ in $k$, write 
$2\Z_k=\p_1^2\p_2$. Then if $v_{\p_1}(\al-1)=1$ for some $\al\in W^+(k)$ we
have $\LL(k,4)=\emptyset$ and $\LL(k,16)\ne\emptyset$, while if
$v_{\p_1}(\al-1)\ge2$ for some $\al\in W^+(k)\setminus W_{(4)}^+(k)$ the 
reverse is true.
\item\label{it_dist_tr} If $2$ is $(1^3)$ in $k$, write $2\Z_k=\p_1^3$. Then 
if $v_{\p_1}(\al-1)=1$ for some $\al\in W^+(k)$ we have $\LL(k,4)=\emptyset$ 
and $\LL_{tr}(k,64)\ne\emptyset$, while if $v_{\p_1}(\al-1)\ge2$ for some 
$\al\in W^+(k)\setminus W_{(4)}^+(k)$ the reverse is true.
\end{enumerate}
\end{proposition}

\begin{remarks}\label{rem_sec7} \hfill \begin{itemize}
\item The conditions in each of the cases \eqref{it_dist_ps}, \eqref{it_dist_pr0}, \eqref{it_dist_pr4}, \eqref{it_dist_tr}
are mutually exclusive, and where we write ``while ... the
reverse is true'', it is indeed for \emph{some} 
$\al\in W^+(k)\setminus W_{(4)}^+(k)$ (it is also true if one replaces 
``some'' by ``all'' but the result would be weaker). In particular, if $\al$ 
and $\be$ in $W^+(k)$ are such that $v_{\p_1}(\al-1)=1$ and 
$v_{\p_1}(\be-1)\ge2$, we have necessarily $\be\in W_{(4)}^+(k)$, since 
otherwise this would contradict the fact that only one of 
$\LL(k,4)$, $\LL(k,16)$, and $\LL_{tr}(k,64)$ is nonempty. 

\item
We can further rephrase the conditions on $\al$ as follows. Let $(\al_i)$ be 
an $\mathbb{F}_2$-basis for $W^+(k)/{k^*}^2=V^+(k)/{k^*}^2=S^+(k)$, 
chosen so that each $\al_i$ is coprime to 2. Then, for any prime $\mfp$ over 
$2$, $v_{\mfp}(\al - 1) = 1$ for some $\al\in W^+(k)$ if and only if it is 
true for one of the $\al_i$.

To prove the nontrivial direction of this we use the equality
\begin{equation}
(\al - 1)(\al' - 1) = (\al \al' - 1) - (\al + \al' - 2)\;,
\end{equation}
so that whenever 
$v_{\mfp}(\al - 1) \geq 2$ and $v_{\mfp}(\al' - 1) \geq 2$, then $v_{\mfp}(\al + \al' - 2) \geq 2$
and so $v_{\mfp}(\al \al' - 1) \geq 2$.
\end{itemize}
\end{remarks}

Before beginning the proof of Proposition \ref{propdist}, we collect some 
useful facts about discriminants.

\begin{proposition}\label{prophas} Let $k$ be a cubic field, write 
$\Disc(k)=Df^2$ with $D$ a fundamental discriminant, and let $p$ be a prime.
\begin{enumerate}\item If $p\ne3$ we have $v_p(f)\le1$ and $p\nmid\gcd(D,f)$.
\item We have $v_3(f)\le2$, and if $v_3(f)=1$ then $3\mid D$.
\item $p$ is totally ramified in $k$ if and only if $p\mid f$. In addition, if $p\ne 3$ then $\big(\frac{D}{p}\big)=\big(\frac{-3}{p}\big)$.

\item $p$ is partially ramified in $k$ if and only if $p\mid D$ and 
$p\nmid f$.\end{enumerate}\end{proposition}

\begin{proof} This is classical, and we refer to \cite{Coh} Section 10.1.5 for
a proof, except for (2) which can be easily deduced from loc.~cit. Note that
the last statement of (2) is true but empty for cyclic cubic fields since
we have $v_3(f)=0$ or $2$.\end{proof}

\begin{proposition}\label{prop_disc_mod} If $k$ is a cubic field then $2$ is
$(1^3)$ in $k$ if and only if $\Disc(k)\equiv20\pmod{32}$, and $(1^21)$ if
and only if $\Disc(k)\equiv8$ or $12\pmod{16}$. In particular we have
$v_2(\Disc(k))\le3$, and we cannot have $\Disc(k)\equiv4\pmod{32}$.
\end{proposition}

\begin{proof} We could prove this by appealing to the Jones-Roberts database of local fields \cite{JR},
as we will do in similar situations later, but here we prefer to give a simple direct argument.
We have $\Disc(k)=Df^2$ for $D$ a fundamental discriminant, and
by Proposition \ref{prophas} we cannot have $p^2\mid f$ or $p\mid\gcd(D,f)$ 
unless possibly $p=3$. Thus if $f$ is even we must have $2\nmid D$ so then 
$v_2(Df^2)=2$, while if $f$ is odd we have $v_2(Df^2)=v_2(D)\le3$. 

By
Proposition \ref{prophas} the prime $2$ is totally ramified
if and only if $2\mid f$. If this happens, since $2^2\nmid f$ we can write $f=2f_1$ with 
$f_1$ odd, so $f^2=4f_1^2\equiv4\pmod{32}$. On the other hand, $D$ is
an odd fundamental discriminant, so $D\equiv1\pmod4$, so it follows already
that $Df^2\equiv4\pmod{16}$. We claim that we cannot have $D\equiv1\pmod8$.
This is in fact a result of class field theory: if $D\equiv1\pmod8$ then $2$ 
is split in $K_2=\Q(\sqrt{D})$ as $2\Z_{K_2}=\p_1\p_2$, so $\p_i\mid f$, 
which is the conductor of the cyclic cubic extension $\widetilde{k}/K_2$, so
by Proposition 3.3.18 of \cite{Coh}, since $\N(\p_i)=2$ we have
$\p_i^2\mid f$, in other words $4\mid f$, a contradiction which proves our
claim.

On the other hand, if $2$ is partially ramified we have $f$ odd so
$f^2\equiv1\pmod8$, and $D$ even, so $D\equiv8$ or $12\pmod{16}$, so
$Df^2\equiv 8$ or $12\pmod{16}$.\end{proof}

\begin{lemma}\label{lem_propdist} \hfill
Let $L$ be a quartic field. If $2$ is unramified in $L$
then $v_2(\Disc(L))=0$, and otherwise:
\begin{itemize}
\item If $2$ splits as $(1^211)$ or $(1^22)$ in $L$ then $v_2(\Disc(L))=2$ or $3$.
\item If $2$ splits as $(1^21^2)$ in $L$ then $v_2(\Disc(L))=4$, $5$, or $6$.
\item If $2$ splits as $(1^31)$ in $L$ then $v_2(\Disc(L))=2$.
\item If $2$ splits as $(2^2)$ in $L$ then $v_2(\Disc(L))=4$ or $6$.
\item If $2$ splits as $(1^4)$ in $L$ then $v_2(\Disc(L))=4$, $6$, $8$, $9$,
$10$, or $11$.
\end{itemize}
\end{lemma}
\begin{proof} The \'etale algebra $L \otimes \Q_2$ splits into one or more 
extensions of $\Q_2$ of degrees totaling to $4$, and $v_2(\Disc(L))$ is the 
sum of the $2$-adic valuations of the discriminants of these extensions. The
Jones-Roberts database \cite{JR} lists all such extensions of $\Q_2$, 
reducing the proof to a simple finite computation.
\end{proof}

\begin{proof}[Proof of Proposition \ref{propdist}]
We assume below that one of $\LL(k,4)$, $\LL(k,16)$, or $\LL_{tr}(k,64)$
contains at least one field $L$, corresponding to an extension $K_6 = k(\sqrt{\al})$
which is ramified for at least one prime over $2$ (abbreviated below to ``is ramified at $2$''). By Theorem \ref{thmcrux} 
we can choose $\al\in W^+(k)$
and $\gd(K_6/k)=4/\c^2$ for the largest $\c\mid (2)$ such that
we can solve $\al\equiv x^2\pmod{^*\c^2}$, and since $L\notin\LL(k,1)$ we have
$\c\ne(2)$, so with the notation of the proposition $\al\notin W_{(4)}^+(k)$.

\begin{lemma} Let $\mfp\mid(2)$ be a prime ideal, $\al\in W(k)$ such that
$v_{\mfp}(\al-1)=1$, and $\c$ the corresponding ideal as above. We then
have $\mfp\nmid\c$.\end{lemma}

\begin{proof} Assume on the contrary that $\mfp\mid\c$, so that
$v_{\mfp}(\al - x^2) \geq 2$ for some $x$. We would then have 
$v_{\mfp}(1 - x^2) = 1$, but either $v_{\mfp}(1-x)=v_{\mfp}(1+x)=0$, or
$v_{\mfp}(1-x)\ge1$ and $v_{\mfp}(1+x)\ge1$, in both cases leading to
a contradiction.\end{proof}

Thus the condition $v_{\mfp}(\al-1)=1$ implies \emph{a priori} that one of 
$\LL(k,4)$, $\LL(k,16)$, or $\LL_{tr}(k,64)$ is nonempty, and in our 
case-by-case analysis we will use the fact that $\mfp \nmid \mfc$ to help 
determine which.
\\
\\
\eqref{it_dist_in}. Suppose that $2$ is inert. Since $\mfc\ne(2)$ we have
then $\mfc = (1)$, so $\N (\mfd (K_6/k)) = 64$, i.e., $L \in \calL(k, 64)$, 
hence $L \in \calL_{tr}(k, 64)$ by Theorem \ref{thmcrux}.
\\
\\
\eqref{it_dist_ps}
Suppose that $2 = \mfp_1 \mfp_2$ is partially split in $k$ with $\mfp_i$ of 
degree $i$. 

If $v_{\mfp_1}(\al - 1) = 1$, then $\mfp_1 \nmid \mfc$ by the above lemma. We 
must also have $\mfp_2 \nmid \mfc$: Otherwise, we would have $\c=\p_2$, hence
$\N(\gd(K_6/k))=4$, so that $v_2(\Disc(L)) = 2$; however,
Theorem \ref{thm_splitting_types} implies that $2$ is $(2^2), (1^2 1^2),$ 
or $(1^4)$ in $L$, and by Lemma \ref{lem_propdist} there is no such quartic 
field for which $v_2(\Disc(L)) = 2$. Therefore 
$\calL_{tr}(k, 64) \neq \emptyset$.

If instead $v_{\mfp_1}(\al - 1) \geq 2$ then $\al\equiv1\pmod{^*\p_1^2}$ so
$\mfp_1 | \mfc$. Since $\al\notin W_{(4)}^+(k)$ we must have
$\mfp_2 \nmid \mfc$, hence $\c=\p_1$, and so $\calL(k, 16) \neq 0$. 
\\
\\
\eqref{it_dist_ts} Suppose that $2=\p_1\p_2\p_3$ is totally split in $k$. Since
$L \not \in \calL(k, 1)$, at least one of the primes above $2$ must ramify in 
$K_6$. By Theorem \ref{thm_splitting_types}, $2$ must be $(2^2)$,
$(1^2 1^2)$, or $(1^4)$ in $L$, and as above Lemma \ref{lem_propdist} implies 
that no such field $L$ has $v_2(\Disc(L)) = 2$, so that 
$\calL(k, 4) = \emptyset$. 

Since $\al$ can be chosen integral and coprime to
$2$ and the $\p_i$ have degree $1$, we have $\al\equiv1\pmod{\p_i}$ for
each $i$, hence $\al=1+2\be$ for some $\be\in\Z_k$. We claim that $\be$ cannot
also be coprime to $2$: indeed, if that were the case, we would in turn have
$\be=1+2\ga$, hence $\al=3+4\ga$, so $\N(\al)\equiv3\pmod{4}$, a contradiction
since $\al$ has square norm. Thus some $\p_i$, say $\p_1$ divides $\be$, hence
$v_{\p_1}(\al-1)\ge2$ so that $\p_1\mid\c$. We cannot have $\c = \p_1\p_j$ for
some $j$ since otherwise $\calL(k,4)$ would not be empty, nor $\c = (2)$ since
$\al\notin W_{(4)}^+(k)$, so that $\c = \p_1$ hence $\calL(k, 16)\ne\emptyset$.
\\
\\
\eqref{it_dist_pr0}. Suppose that $2 = \mfp_1^2 \mfp_2$ is partially ramified 
with $\Disc(k) \equiv 0 \ (\textmod \ 8)$, so by Lemma \ref{prop_disc_mod} we 
have $v_2(\Disc(k)) = 3$. Since Lemma \ref{lem_propdist} implies that there is
no quartic field $L$ with $v_2(\Disc(L)) = 7$, it follows that
$\calL(k, 16) = \emptyset$. In particular we cannot have $\mfc=\mfp_1$ or
$\mfc=\mfp_2$.

If $v_{\mfp_1}(\al - 1) = 1$, then again $\mfp_1 \nmid \mfc$ for the corresponding $K_6$, and we also have $\mfp_2 \nmid \mfc$ (otherwise $\mfc=\mfp_2$)
so $\calL_{tr}(k, 64) \neq \emptyset$.

If $v_{\mfp_1}(\al - 1) \geq 2$, then $\mfp_1 | \mfc$. We cannot have
$\p_1^2\nmid\c$ and $\p_2\nmid\c$ since otherwise $\c=\p_1$. Since 
$\al\notin W_{(4)}^+(k)$ we thus have $\c=\mfp_1^2$ or $\mfp_1\mfp_2$,
so $\calL(k, 4) \neq \emptyset$.
\\
\\
\eqref{it_dist_pr4}.
Suppose that $2 = \mfp_1^2 \mfp_2$ is partially ramified with 
$\Disc(k) \equiv 4 \ (\textmod \ 8)$. We use here some results from Section
\ref{sec:CDO}. By Definition \ref{def_zkc} (more precisely Theorem
\ref{prop_compute_zkc}), we have $z_k(\p_2)=2$, so by Lemma
\ref{easyzk1} we have $|Z_{\p_2}[N]/Z_{\p_2}^2|=1$, hence by the
exact sequence of Proposition \ref{prop:exseq} we deduce that
$S_{\p_2^2}[N]=S[N]$. By definition of $\c$ this implies that $\p_2\mid\c$.

If $v_{\mfp_1}(\al - 1) = 1$, then $\mfp_1 \nmid \mfc$ so $\c=\p_2$ and
$\calL(k, 16) \neq \emptyset$.
If $v_{\mfp_1}(\al - 1) \geq 2$, then $\mfp_1 | \mfc$, so since
$\al\notin W_{(4)}^+(k)$ we have $\c=\p_1\p_2$, hence 
$\calL(k, 4) \neq \emptyset$.
\\
\\
\eqref{it_dist_tr}. Suppose that $2 = \mfp^3$ is totally ramified. If 
$v_{\mfp}(\al - 1) = 1$, then $\p\nmid\c$ so $\c = (1)$ and
$\calL_{tr}(k, 64) \neq \emptyset$. Thus assume that 
$v_{\mfp}(\al - 1) \geq 2$. We first claim that there exists
$\ga\in k^*$ such that $v_{\mfp}(\al\ga^2-1)\ge3$. Indeed, set $\ga=1+\pi u$,
where $\pi$ is a uniformizer of $\p$ and $u$ is $2$-integral. We have
$\ga^2=1+2\pi u+\pi^2 u^2\equiv 1+\pi^2 u^2\pmod{\p^4}$, so
$v_{\p}(\al\ga^2-1)\ge3$ is equivalent to $u^2\equiv (1/\al-1)/\pi^2\pmod{\p}$
which has a solution (for instance $u=(1/\al-1)/\pi^2$), proving our claim.
We now claim that for any $\be$ of square norm and coprime to $2$ (such as 
$\al\ga^2$), we cannot have $v_{\p}(\be-1)=3$: Indeed, assume this is the case,
so that $\be=1+2v$ with $v$ coprime to $2$. By expanding we see that 
$\N(\be)\equiv1+2\Tr(v)\pmod{4}$. As in the proof of Theorem
\ref{prop_compute_zkc}, we note that the different $\Gd(k)$ is divisible by $\p^2$,
so that $\p=2\p^{-2}\subset2\Gd^{-1}(k)$, hence $2\mid\Tr(w)$ for any $w\in\p$.
Since $v$ is coprime to $2$ we have $v=1+w$ with $w\in\p$, so we deduce that
$\Tr(v)$ is odd, hence that $\N(\be)\equiv3\pmod{4}$, contradicting the fact
that $\be$ has square norm and proving our claim. It follows that
$v_{\p}(\al\ga^2-1)\ge4$, so that $\p^2\mid\c$, and since
$\al\notin W_{(4)}^+(k)$ we have $\mfc = \mfp^2$, so $\calL(k, 4) \neq 0$.
\end{proof}

\begin{proof}[Proof of Proposition \ref{prop_split_table}]
The possibilities listed for $L \in \LL(k, 1)$ (i.e., with $n^2 = 1$) are precisely those 
corresponding to the possibilities allowed by the table in Theorem 
\ref{thm_splitting_types}. In particular we must have $K_6/k$ unramified, and for each 
row not ruled out we found an example by computer search.

For $\LL(k, 4)$, $\LL(k, 16)$, and $\LL_{tr}(k, 64)$, we may rule out the 
following possibilities because they correspond to $K_6/k$ unramified: $k=(3), L=(31)$;
$k = (21), L=(4), (211)$; $k=(111), L=(22), (1111)$; 
$k=(1^21), L=(1^211), (2 1^2)$; $k=(1^3), L=(1^31)$. 
Also, for each splitting type in $k$ we rule out columns as in Proposition 
\ref{propdist}.

By definition, we can rule out all possibilities for $\LL_{tr}(k, 64)$ for
which $2$ is not $(1^4)$ in $L$.

We can rule out additional possibilities based on discriminant mismatches. Note
that $\gd(K_6/k)=4/\c^2$ for some $\c\mid (2)$, so that if a prime ideal
$\p$ of $k$ ramifies in $K_6/k$ we have $\p^2\mid\gd(K_6/k)$, hence
the product of $\N(\p)^2$ over all $\p$ ramified in $K_6/k$ divides
$\N(\gd(K_6/k))=\Disc(L)/\Disc(k)$. For example, suppose that
$2$ is $(111)$ in $k$ and $(1^4)$ in $L$. By Theorem \ref{thm_splitting_types},
$2$ must be $(1^2 1^2 1^2)$ in $K_6$, so $64\mid N(\mfd(K_6/k))$, in other
words $L\in\LL_{tr}(k, 64)$, contradicting Proposition \ref{propdist} which
tells us that $L \in \LL(k, 16)$.

Variants of this argument rule out the following cases: 
$k = (21), L=(1^4), L \in \LL(k, 16)$; $k=(1^21), L=(1^4), L \in \LL(k, 4)$.

If $k = (1^2 1)_4$, $L = (1^4)$, $L \in \LL(k, 16)$, we have 
$\Disc(L) = 16 \cdot \Disc(k) \equiv 3 \cdot 2^6 \ (\textmod \ 2^8)$ by 
Proposition \ref{prop_disc_mod}. However, by the Jones-Roberts database, there
are three totally ramified quartic $L_v /\Q_2$ with $v_2(\Disc(L_v)) = 6$, 
and they all three satisfy $\Disc(L_v) \equiv 2^6 \ (\textmod \  2^8)$.
(The discriminant of $L_v$ is defined up to squares of $2$-adic units, so that
this equation is well-defined.) Therefore we cannot have 
$L \otimes \Q_2 \simeq L_v$ so this case is impossible.

If $k = (1^21)_4, L=(1^21^2), L \in \LL(k, 4)$, we have
$\Disc(L)\equiv 48\pmod{64}$ by Proposition \ref{prop_disc_mod}.
However, since $2$ is $(1^2 1^2)$ in $L$, then $L \otimes \Q_2$ is a product 
of two quadratic extensions of $\Q_2$, each of whose discriminants must have 
$2$-adic valuation $2$ (since $v_2(\Disc(L))=4$), and 
therefore (by the local analogue of Proposition \ref{prop_disc_mod}, or
again by the Jones--Roberts database which here is trivial) each of 
whose discriminants must be $12\pmod{16}$. Therefore,
$\Disc(L)\equiv (12\bmod{16}) \cdot (12\bmod{16}) \equiv 16 \pmod {64}$, 
contradicting the above.

The cases not ruled out above can all happen; to prove this we found the 
examples listed in the table by computer search.
\end{proof}

\section{Proof of Theorem \ref{thm_main}}\label{sec_quartic_proofs}

We begin with Theorem \ref{thma4s4}, which gave an expression for $\Phi_k(s)$ as a sum involving
the product
$F_k(\chi, s)$, defined in \eqref{eqn_a4s4_euler} by

\begin{equation}\label{eqn_a4s4_euler_repeat}
F_k(\chi,s)=
\prod_{p\Z_k=\p_1\p_2}\left(1+\dfrac{\chi(\p_2)}{p^s}\right)\prod_{p\Z_k=\p_1^2\p_2}\left(1+\dfrac{\chi(\p_1\p_2)}{p^s}\right)
\prod_{p\Z_k=\p_1\p_2\p_3}\left(1+\dfrac{\chi(\p_1\p_2)+\chi(\p_1\p_3)+\chi(\p_2\p_3)}{p^s}\right)\;,\end{equation}
It remains to explicitly evaluate and sum the contributions of $F_k(\chi, s)$ for each $\c$ and $\chi$.
We begin by evaluating the contribution of the trivial characters,
which is straightforward.

To handle the nontrivial characters, we must apply 
Propositions \ref{prop_c4} and \ref{prop_c4_ext} to reinterpret them as characters of class groups, and class field
theory to further reinterpret them as characters of Galois groups, after which we can evaluate the $\chi(\mfp_i)$
in terms of the splitting of $\mfp_i$ in quadratic extensions $K_6$ of $k$, which \eqref{eqn_artin} relates to the splitting of $p$
in the associated quartic extensions $L/\Q$. 

We are then ready to evaluate the contributions of the nontrivial characters. When $|C_{(4)}| = |C_{(1)}|$, which
includes the $A_4$ case, this is quite straightforward.
When $|C_{(4)}| > |C_{(1)}|$, we must determine the set of $\c | (2)$ for which $|C_{\c^2}| = |C_{(4)}|$; this relies on our
work in Section \ref{sec_arith_s4} and we carry out the computation 
in Proposition \ref{prop_list_c}. In either case, our evaluation of these contributions yields the formulas of
Theorem \ref{thm_main} and
finishes the proof.

\subsection{Evaluating the contribution of the trivial characters}
The contribution of the trivial characters is equal to 
\begin{equation}\label{eqn_s4_trivial_cont}
\dfrac{S(s)}{2^{r_2(k)}}\prod_{p\Z_k=\p_1\p_2,\ p\ne2}\left(1+\dfrac{1}{p^s}\right)\prod_{p\Z_k=\p_1^2\p_2,\ p\ne2}\left(1+\dfrac{1}{p^s}\right)\prod_{p\Z_k=\p_1\p_2\p_3,\ p\ne2}\left(1+\dfrac{3}{p^s}\right)\;,
\end{equation}
where
$$S(s)=\dfrac{1}{2^{3s-2}}\sum_{\c\mid2\Z_k}\N\c^{s-1}z_k(\c)\prod_{\p\mid\c}\left(1-\dfrac{1}{\N\p^s}\right)T_{\c,2}(s)\;,$$
for suitable $T_{\c,2}(s)$ detailed below and $z_k(\c)$ as in Definition \ref{def_zkc}.

We distinguish all the possible splitting types of $2$ in $k$ as above.
\smallskip

\begin{enumerate}
\item $2\Z_k=\p_1\p_2$, $\p_2$ of degree $2$. Here $T_{\c,2}(s)=1+1/2^s$ if 
$\c=\Z_k$ or $\c=\p_1$, and $T_{\c,2}(s)=1$ otherwise. Thus,
\begin{align*}S(s)&=(1/2^{3s-2})\big(1+1/2^s+2^{s-1}(1-1/2^s)(1+1/2^s)+2^{2s-2}(1-1/2^{2s}\big)\\
&\phantom{=}+2^{3s-2}(1-1/2^s)(1-1/2^{2s}))=1+1/2^{2s}+4/2^{3s}+2/2^{4s}\;.
\end{align*}
The remaining computations
of $S(s)$ are exactly similar and so we omit the details.
\item $2\Z_k$ is inert. Then $T_{\c, 2}(s) = 1$ for all $\c$.
\item $2 = \p_1 \p_2 \p_3$. Then $T_{\c,2}(s)=1+1/2^s$ if 
$\c=\Z_k$ or $\c=\p_1$, and $T_{\c,2}(s)=1$ otherwise.
\item $2\Z_k=\p_1^2\p_2$. Here $T_{\c,2}(s)=1+1/2^s$ for $\c=\Z_k$ and $1$ 
otherwise. The two different values of $S(s)$ for $\Disc(k)\equiv0$ or
$4\pmod{8}$ come from the different values of $z_k(\c)$.
\item $2\Z_k=\p_1^3$, i.e., $2$ totally ramified. Here $T_{\c,2}(s)=1$ for all
$\c$.
\end{enumerate}

In all cases we compute that $S(s) = M_1(s)$ as given in
Theorem \ref{thm_main},
so that 
\eqref{eqn_s4_trivial_cont} is equal to the first term of $\Phi_k(s)$, with 
equality if and only if $k$ is totally real and $2 \nmid h_2^+(k)$.
\smallskip

\subsection{Interpreting the nontrivial characters in terms of class groups}
Unless $k$ is totally real and $2 \nmid h_2^+(k)$, 
Propositions \ref{prop_c4} and \ref{prop_l2k} imply that there
are also nontrivial characters. 

We discuss the case $\c = 1$ first. 
Proposition \ref{prop_c4}
gives an isomorphism $\phi: \Cl(k)/\Cl(k)^2 \rightarrow C_{(1)}$, and we may write 
a character $\chi$ of $C_{(1)}$ as a character $\chi_{\phi}$ of $\Cl(k)/\Cl(k)^2$, where
$\chi_{\phi}(\mfa) = \chi(\phi(\mfa))$, so that $\chi(\mfb) = \chi_{\phi}(\phi^{-1}(\mfb))
= \chi_{\phi}(\mfb / \sqrt{ \N(\mfb)})$. We further use the Artin map of class field theory
to rewrite $\chi_{\phi}$ as a character of the quadratic field determined by $\Ker(\chi_{\phi})$,
so that $\chi_{\phi}(\mfp) = 1$ if $\mfp$ splits in this quadratic field, and $\chi_{\phi}(\mfp) = -1$
if $\mfp$ is inert.

The set of these characters corresponds precisely to the set of unramified quadratic extensions of $k$,
i.e., fields in $\calL(k, 1)$. By 
Theorem \ref{thm_quartic_corr} there are $a(k)$ characters for each field in $\calL(k, 1)$,
where $a(k)$ is equal to $3$ or $1$ in the $A_4$ and $S_4$ cases respectively.

In the $A_4$ case, or if otherwise $C_{(1)} \simeq C_{(4)}$,
Propositions \ref{prop_c4} and \ref{prop_53} imply that the natural surjection $C_{\c^2} \rightarrow C_{1}$
is an isomorphism for each $\c$, and so may we regard each character $\chi$ of $C_{\c^2}$ first as
a character of $C_{(1)}$, and then as a character of a quadratic field as above, provided that we still write
$\chi(\mfa) = 0$ if $\mfa$ is not coprime to $\mfc$.

If $C_{(1)} \not \simeq C_{(4)}$, then each $C_{\c^2}$ will be naturally isomorphic to either $C_{(1)}$ or $C_{(4)}$, 
and in Proposition \ref{prop_list_c} we determine which on a case-by-case basis. 
Those $C_{\c^2}$ isomorphic to $C_{(1)}$ are handled as before. Those
$C_{\c^2}$ isomorphic to $C_{(4)}$ may be handled similarly: 
composing this isomorphism with $\phi$ we obtain
an isomorphism with 
$\Cl_{(4)}(k) / \Cl_{(4)}(k)^2$, and we obtain a character associated to a quadratic field in $\calL_2(k)$,
with $\chi(\mfa) = \chi_{\phi}(\mfa/\sqrt{\N(\mfa)})$, except when
$(\mfa, \mfc) \neq 1$ in which case we have $\chi(\mfa) = 0$.

This latter construction does not allow us to compute $\chi(\mfa)$ when 
$\mfa$ is coprime to $\c$ but not $(2)$, and we will need to do this 
in two cases where $C_{\c^2} \isom C_{(4)}$. Here we apply 
Proposition \ref{prop_c4_ext} to extend $\phi$ to
an isomorphism $\Cl_{4/\c^2}(k)/\Cl_{4/c^2}(k)^2 \rightarrow C_{\c^2}$ which agrees with the $\phi$ given previously
on ideals coprime to $(2)$.

Putting this all together, for each $\mfc | (2)$ we may thus interpret the sum over nontrivial characters in $X_{\c^2}$ 
as a sum
over all quadratic extensions $K_6/k$,
unramified at infinity, and with either $\mfd(K_6/k) = \Z_k$ or $\mfd(K_6/k) | (4)$ as appropriate. By Corollary \ref{corcrux}
these correspond to 
quartic fields in $\calL(k, 1)$ or $\calL_2(k)$ respectively.

\subsection{Evaluating the contribution of the nontrivial characters.}
We begin with the contributions of fields $L \in \calL(k, 1)$, which correspond to 
characters of all the groups $C_{\c^2}$ occurring in Theorem \ref{thma4s4}.

For each $L \in \calL(k, 1)$, the contribution of the nontrivial characters is
\begin{multline}\label{eqn_s4_cont}
\dfrac{S'(s)}{2^{r_2(k)}}
\prod_{p\Z_k=\p_1\p_2,\ p\ne2}\left(1+\dfrac{ \chi_{\phi}(\mfp_1)}{p^s}\right)
\prod_{p\Z_k=\p_1^2\p_2,\ p\ne2}\left(1+\dfrac{ \chi_{\phi}(\mfp_1)
}{p^s}\right) \times \\ 
\prod_{p\Z_k=\p_1\p_2\p_3,\ p\ne2}\left(1+\dfrac{ \chi_{\phi}(\mfp_3) + \chi_{\phi}(\mfp_2) + \chi_{\phi}(\mfp_1)
}{p^s}\right),
\end{multline}
where $S'(s)$ is a sum over the ideals dividing $2 \Z_k$ as before, and $\chi_{\phi}(\mfp)$ is $1$ or $-1$ depending
on whether $\mfp$ splits or is inert in the quadratic extension $K_6/k$ corresponding to $L$. 

For the three splitting types of $p$ in $k$ occurring in \eqref{eqn_s4_cont},
Theorem \ref{thm_splitting_types} gives the following
possible splitting types of $p$ in $K_6$ and $L$:
\begin{itemize}
\item
$(21)$ in $k$, $(42)$ in $K_6$, $(4)$ in $L$. Then $\chi_{\phi}(\p_1) = -1$.
\item
$(21)$ in $k$, $(2211)$ in $K_6$, $(211)$ in $L$. Then $\chi_{\phi}(\p_1) = 1$.
\item
$(111)$ in $k$, $(2211)$ in $K_6$, $(22)$ in $L$. Then 
$\chi_{\phi}(\p_3) + \chi_{\phi}(\p_2)  + \chi_{\phi}(\p_1)  = -1$.
\item
$(111)$ in $k$, $(111111)$ in $K_6$, $(1111)$ in $L$. Then 
$\chi_{\phi}(\p_3) + \chi_{\phi}(\p_2)  + \chi_{\phi}(\p_1)  = 3$.
\item
$(1^21)$ in $k$, $(2^211)$ in $K_6$, $(2 1^2)$ in $L$. Then 
$\chi_{\phi}(\p_1) = -1$.
\item
$(1^21)$ in $k$, $(1^21^211)$ in $K_6$, $(1^2 11)$ in $L$. Then 
$\chi_{\phi}(\p_1) = 1$.
\end{itemize}

These match the values of $\om_L(p)$ given in Definition \ref{def_omega}, as 
required.

Once again we have
\begin{equation}\label{eqn_def_sp}
S'(s) =\dfrac{1}{2^{3s-2}}\sum_{\c}\N\c^{s-1}z_k(\c)\prod_{\p\mid\c}\left(1-\dfrac{1}{\N\p^s}\right)T_{\c,2}(s)\;,
\end{equation}
where $\c$ ranges over all ideals dividing $2 \Z_k$, and $T_{\c, 2}(s)$ 
now depends on the splitting of
$2$ in $K_6$. 
For example, if $2$ is totally split in $\Z_k$ as $2 \Z_k = \p_1 \p_2 \p_3$ and $K_6/k$
is unramified, we check that $T_{\c, 2}(s)$ is equal to $1 + \omega_L(2)/2^s$ for $\c = \Z_k$, and
$1 \pm \frac{1}{2^s}$ for $\c = \p_i$, depending on whether $\p_i$ is split or inert in $K_6/k$.
The proof that $S'(s) = M_{2,L}(s)$ breaks up into six cases 
depending on the splitting type of $2$ in $k$, 
and, when $2$ is partially ramified, $\Disc(k) \pmod 8$. 
The computation
is similar to our previous computations and we omit the details.
Applying this computation with the $n^2 = 1$ entries of Proposition \ref{prop_split_table}
and the values of $\chi_{\phi}(\mfp_i)$ just computed, we obtain the $n^2 = 1$ entries
in the table of Theorem \ref{thm_main}.

Recall from Propositions \ref{prop_c4} and \ref{prop_l2k} that either $|C_{4}| = |C_{1}|$ or 
$|C_{(4)}| = 2|C_{(1)}|$, hence either
$|\calL_2(k)| = |\calL(k, 1)|$ or $2|\calL(k, 1)| + 1$. If
$|\calL_2(k)| = |\calL(k, 1)|$, and in particular in the $A_4$ case, the proof of Theorem \ref{thm_main} is
now
complete.

\subsection{The case where  $|\calL_2(k)| = 2|\calL(k, 1)| + 1$}\label{subsec_unequal}
We henceforth assume
that $|\calL_2(k)| = 2|\calL(k, 1)| + 1$, in which case we must also compute the 
contributions from the fields in $\calL_2(k) \backslash \calL(k, 1)$. $F_k(\chi, s)$ is evaluated in essentially
the same way, 
but
in \eqref{eqn_def_sp}
we sum over only those $\c$ for which $|C_{\c^2}| = |C_{(4)}|$.
To compute $S'(s)$ we must therefore compute a list of such $\c'$.

\begin{proposition}\label{prop_list_c}
Assume that $|C_{(4)}| = 2|C_{(1)}|$ and $|\calL_2(k)| = 2|\calL(k, 1)| + 1$, 
so that exactly one of $\calL(k, 4)$, $\calL(k, 16)$, or $\calL_{tr}(k, 64)$
is nonempty. Then, for $\c | (2)$, $|C_{\c^2}| = 2|C_{(1)}|$ if and only if 
$\c$ is one of the following ideals:

If $\calL(k, 4)$ is nonempty:
\begin{itemize}
\item
If $2$ is $(1^2 1)_0$, $\mfc = \p_1, \p_2, \p_1^2, \p_1 \p_2, (2)$.
\item
If $2$ is $(1^2 1)_4$, $\mfc = \p_1, \p_1^2, \p_1 \p_2, (2)$.
\item
If $2$ is $(1^3)$, $\mfc = \mfp, \mfp^2, (2)$.
\end{itemize}

If $\calL(k, 16)$ is nonempty:
\begin{itemize}
\item
If $2$ is $(21)$, $\mfc = \mfp_1, \mfp_2, (2)$.
\item
If $2$ is $(111)$, $\mfc = \p_1, \mfp_1 \mfp_2, \mfp_1 \mfp_3, \mfp_2 \mfp_3, (2)$.
The distinguished prime ideal $\p_1$ is the prime ideal not ramified in 
$K_6/k$ for $K_6$ corresponding to $\calL(k, 16)$.
\item
If $2$ is $(1^2 1)\ ($necessarily $(1^2 1)_4)$, $\mfc = \mfp_1^2, (2)$.
\end{itemize}

If $\calL_{tr}(k, 64)$ is nonempty:
\begin{itemize}
\item
If $2$ is $(3)$, $\mfc = (2)$.
\item
If $2$ is $(21)$, $\mfc = \p_2, (2)$.
\item
If $2$ is $(1^2 1) \ ($necessarily $(1^2 1)_0)$, $\mfc = \p_1^2, (2)$.
\item
If $2$ is $(1^3)$, $\mfc = (2)$.
\end{itemize}
\end{proposition}

\begin{proof}
First observe from \eqref{it_all_iso} of Proposition \ref{prop_c4} that if some $\c$ appears on the list, so do all of its
multiples, and conversely if some $\c$ does not, neither do its factors.

As the size of $C_{\mfc^2} = \Cl_{\mfc^2}[N] / D_{\mfc^2}[N]$ is measured by the 
exact sequence of Proposition \ref{prop:exseq}, we consider the change in the size
of the other factors when we replace $\mfc$ by a multiple $\mfc'$ of $\mfc$. Two
other factors may change in \eqref{eqn:exseq}:
\begin{itemize}
\item The quantity $Z_{\mfc}[N]/Z_{\mfc}^2$ may increase in size; this is necessary but not
sufficient for $C_{\mfc^2}$ to increase in size. By Proposition \ref{easyzk1},
$|Z_{\mfc}[N]/Z_{\mfc}^2| = \N(\mfc) / z_k(\mfc)$ where $z_k(\mfc)$ is defined in Definition \ref{def_zkc} 
(see also Theorem \ref{prop_compute_zkc}). 
\item
The quantity $S_{\mfc^2}[N]$ may decrease in size.  Then 
$Z_{\mfc}[N]/Z_{\mfc}^2$ must increase, and by a larger proportion if $|C_{\mfc^2}|$ also increases.

To check this possibility, observe that 
Proposition \ref{prop:hec} implies that the equality $S_{\mfc^2}[N] = S_{\mfc'^2}[N]$ with $\mfc | \mfc'$
is equivalent to the fact that
all quadratic extensions of $k$, unramified at infinity,
whose discriminants divide $4/\mfc^2$ must in fact have discriminants dividing $4/\mfc'^2$.

\end{itemize}

This information is enough to prove the proposition; we explain in detail for a representative case,
and give a sketch for the remaining cases.

Suppose that $(2) = \mfp_1^2 \mfp_2$ in $L$ with $d = \Disc(k) \equiv 0 \ (\textmod \ 8)$. By Definition \ref{def_zkc},
$z_k(\mfc)$ doubles when
$\mfc$ increases from $\mfp_1^2$ to $(2)$, or from $\mfp_1$ or $\mfp_2$ to $\mfp_1 \mfp_2$. $\N(\mfc)$ also doubles
in each of these cases and so $|C_{\mfc^2}|$ stays the same.

There are quartic fields of discriminant $4d$ or $64d$, but not both. In the former case, 
$|S_{\mfc^2}[N]|$ decreases when $\mfc$ goes to $(2)$ from either $\mfp_1 \mfp_2$ or $\mfp_1^2$, or both. 
We can rule
out $\mfp_1^2$, because $\N(\mfp_1^2) / z_k(\mfp_1^2) = \N(2) / z_k((2))$.
Therefore $|S_{(\mfp_1 \mfp_2)^2}[N]| = 2|S_{(4)}[N]|$, and so 
$|C_{(\mfp_1 \mfp_2)^2}| = |C_{(4)}|$.
Put together, these observations establish that $|C_{\mfc^2}|$ is the same
for all $\mfc \in \{ \mfp_1, \mfp_2, \mfp_1^2, \mfp_1 \mfp_2, (2) \}$,
as claimed.\footnote{Moreover, observe that this tells us that the associated quadratic extensions
$K_6 = k(\sqrt{\alpha})$ satisfy $\mfd(K_6/k) = \p_1^2$. This adds information to what we determined
in the proof of Proposition \ref{propdist}; conversely, in cases where we determined 
this information in Proposition \ref{propdist}, we could apply it in the present proof (although in each case
it can be determined from the statement of Proposition \ref{propdist} and the considerations described above).}

If instead there are quartic fields of discriminant $64d$, then
$|S_{\mfc^2}[N]|$ decreases when $\mfc$ goes from $1$ to either $\mfp_1$ or 
$\mfp_2$. As $\N(\c)/z_k(\c)$ doubles, $|C_{\c^2}|$ stays the same. Therefore, 
$|C_{\mfc^2}|$ is the same for $\mfc \in \{ \mfp_1^2, (2) \}$,
and separately for all $\mfc \in \{ 1, \mfp_1, \mfp_2, \mfp_1 \mfp_2 \}$, 
proving this case of the proposition.
\\
\\
The remaining cases are similar, and we omit the details. Some brief remarks, helpful for verifying the results:
\begin{itemize}
\item
When $2$ is $(1^21)$ in $k$ the values of $z_k(\mfc)$ depend on whether 
$\Disc(k)$ is $0$ or $4 \ (\textmod \ 8)$, explaining the different results for these two cases.
\item
When $2$ is $(21)$ in $k$ and $\calL_{tr}(k, 64) \neq \emptyset$, 
$|S_{\c^2}[N]|$ decreases when $\c$ increases from $\Z_k$ to $\p_2$, but 
$|\N(\c)/z_k(\c)|$ increases by a factor of $4$ and we cannot conclude that 
$|C_{\c^2}|$ remains the same. Indeed, $\N(\p_2)/z_k(\p_2) = \N(2)/z_k((2))$,
and $2\N(1)/z_k((1)) = \N(\p_1)/z_k(\p_1)$ with a decrease in $|S_{\c^2}[N]|$,
so that $|C_{\p_2^2}| = |C_{(4)}|$ and
$|C_{(1)}| = |C_{\p_1^2}|$, so that by elimination we deduce that
$|C_{\p_1^2}| \neq |C_{(4)}|$.
\item
When $2$ is $(111)$, $\calL(k, 16)$ is nonempty, and the corresponding 
extensions $K_6/k$ must all have the same discriminant, corresponding to an 
increase of $1$ in the $2$-rank of appropriate ray class groups. This 
discriminant is $(\p \p')^2$ for two of the primes $\p, \ \p'$ of $k$ above
$2$, and we write $\mfp_1$ for the remaining prime.
\end{itemize}
\end{proof}

For each $L$ in $\calL(k, 4)$, $\calL(k, 16)$, and $\calL_{tr}(k, 64)$, we can
now prove that $S'(s) = M_{2,L}(s)$. The analysis again breaks up into 
cases, and we present the details for a representative case.

Suppose then that $L \in \calL(k, 16)$ and that $2$ is totally split in $L$. As before write $\p_1$ for the prime of $k$
above $2$ which does not ramify in the extension $K_6/k$ corresponding to $L$, and write $\p_2, \ \p_3$ for the other
two primes above $2$. By Theorem \ref{thma4s4} we have
\begin{equation}
T_{\c, 2}(s) = \frac{1 + \chi(\p_1 \p_2) + \chi(\p_1 \p_3) + \chi(\p_2 \p_3)}{2^s},
\end{equation}
where for each $\c$, $\chi(\p_i \p_j) = \chi_{\phi}(2/\p_i \p_j)$ if $\p_i \p_j$ is coprime
to $\c$, and  $\chi(\p_i \p_j) = 0$ otherwise. In particular, we have $T_{\c, 2}(s) = 1 + \frac{\chi_{\phi}(\p_1)}{2^s}$
when $\c = \p_1$, and $T_{\c, 2}(s) = 1$ for the other $\c$ listed in Proposition \ref{prop_list_c}.
By Proposition \ref{prop_list_c} we obtain contributions to $S'(s)$ from
the ideals
$\mfc = \p_1, \mfp_1 \mfp_2, \mfp_1 \mfp_3, \mfp_2 \mfp_3, (2)$ as follows:

\begin{itemize}
\item
$\mfc = \mfp_1: \ \ \ 
 \dfrac{4}{8^s} \cdot 
\dfrac{2^s}{2} \cdot 
\left(1 - \dfrac{1}{2^s}\right)\left(1 + \dfrac{ \chi_{\phi}(\mfp_1) }{2^s} \right)
= 
2 \cdot 2^{-2s} + 2(\chi_{\phi}(\mfp_1) - 1) 2^{-3s} - 2 \chi_{\phi}(\p_1) 2^{-4s}\;.
$
\item $\mfc = \mfp \mfp': \ \  \dfrac{4}{8^s} \cdot 
\dfrac{4^s}{4} \cdot 
\left(1 - \dfrac{1}{2^s}\right)^2
=
2^{-s} - 2 \cdot 2^{-2s} + 2^{-3s}\;.
$
\item $\mfc = (2): \ \  \dfrac{4}{8^s} \cdot 
\dfrac{8^s}{4} \cdot 
\left(1 - \dfrac{1}{2^s}\right)^3
=
1 - 3 \cdot 2^{-s} + 3 \cdot 2^{-2s} - 2^{-3s}\;.
$
\end{itemize}
Adding each of these contributions (with three ideals of the form $\mfp \mfp'$), we obtain
a total of
$1 - 2^{-2s} + 2 \chi_{\phi}(\mfp) 2^{-3s} - 2 \chi_{\phi}(\mfp) 2^{-4s}.$ If $L \in \calL(k, 16)$, then
$\mfd(K_6/k) = (\mfp_2 \mfp_3)^2$ and $\mfp_1$ can split or be inert in $K_6$. If it splits, then 
$\chi_{\phi}(\mfp) = 1$, and 
Theorem \ref{thm_splitting_types} implies that the splitting type of $2$ in $L$ is $(1^2 1^2)$, and the above contribution
is $1 - 2^{-2s} + 2 \cdot 2^{-3s} - 2 \cdot 2^{-4s}$; if it is inert, then 
Theorem \ref{thm_splitting_types}  implies that the splitting type of $2$ in $L$ is $(2^2)$, and the above contribution
is $1 - 2^{-2s} - 2 \cdot 2^{-3s} + 2 \cdot 2^{-4s}$.

This verifies that $S'(s) = M_{2,L}(s)$ for this case, and 
the remaining cases are proved in the same way. The only remaining case
where $\c \neq 1$ is $T_{\c, 2}(s) = 1 + \chi_{\phi}(\mfp_1)/2^s$
when $L \in \calL(k, 16)$, $2$ is partially split,
and $\mfc = \p_1$, where once again $C_{\c^2} \isom \Cl_{4/\c^2}/\Cl_{4/\c^2}^2$.
In the other cases $T_{\c, 2}(s) = 1$
and the computations are simpler; in particular, it suffices to appeal to Proposition \ref{prop_c4} rather than
Proposition \ref{prop_c4_ext}.

For each combination of possibilities for $n^2$ and the splitting types of $K_6$ and $L$ listed in
Proposition \ref{prop_split_table}, we thus check that $S'(s) = M_{2, L}(s)$, and the product over primes
$p \neq 2$ is handled as in \eqref{eqn_s4_cont}. This completes the proof.

\section{Results with Signatures}\label{sec_sig}

In the case of cubic fields with given quadratic resolvent, the quadratic
resolvent determines the signature of the cubic field. In our case this is
not true: if $k$ is a complex cubic field then the corresponding quartic
fields $K$ have signature $(2,1)$, but if $k$ is a totally real cubic
then $K$ can either be totally real (signature $(4,0)$), or totally complex
(signature $(0,2)$), and we may want to separate these families.

As mentioned in \cite{Coh_a4s4} and \cite{CDO_quartic}, it is possible to modify
the above work to take into account signature constraints (or more generally
a finite number of local conditions). 
Since there are no new results when $k$ is complex, in this section we always
assume that $k$ is a totally real cubic field.

We define $\calF^+(k)$ as the subset of all totally real quartic fields in
$\calF(k)$, and define
$$\Phi_k^+(s)=\dfrac{1}{a(k)}+\sum_{K\in\FF^+(k)}\dfrac{1}{f(K)^s}\;.$$
We define $\calL_2^*(k)$ and $\calL^*(k, 1)$ as in Definition \ref{defll},
only without any restriction that the quartics be totally real.
In this setting we have the following:
\begin{theorem}\label{thm_main_sig}
The formulas of Theorem \ref{thm_main} hold for $\Phi_k^+(s)$ with the following two
modifications:
\begin{itemize}
\item The formulas are multiplied by $\frac{1}{4}$.
\item The sums over $\calL_2(k)$ and $\calL(k, 1)$ are replaced with sums over
$\calL_2^*(k)$ and $\calL^*(k, 1)$.
\end{itemize}
\end{theorem}
This result exhibits a curious duality: in order to enumerate quartic fields {\itshape with}
signature conditions, we sum over fields in $\calL_2^*(k)$ {\itshape without} signature conditions.

\subsection{Theorem \ref{thma4s4} with signature conditions}

We first sketch a proof of a version of 
Theorem \ref{thma4s4} for this setting, 
mainly
explaining the difference with the case where no signature conditions are
added.

\begin{theorem}\label{thma4s4plus}
$$\Phi_k^+(s)=\dfrac{1}{a(k)2^{3s}}\sum_{\c\mid2\Z_k}\N\c^{s-1}z_k(\c)\prod_{\p\mid\c}\left(1-\dfrac{1}{\N\p^s}\right)\sum_{\chi\in X^+_{\c^2}}F_k(\chi,s)\;,$$
where $X^+_{\c^2}$ is the group of characters of $C^+_{\c^2}$, defined as
$C_{\c^2}$ with the added condition that $\be$ be totally positive,
and $z_k(\c)$ and $F_k(\chi,s)$ are as in Theorem \ref{thma4s4}.\end{theorem}

\begin{proof} 
The condition that $K$ is totally real is equivalent to the fact that the
corresponding quadratic extension of trivial norm $K_6/k$ is unramified at 
infinity, or equivalently that $K_6=k(\sqrt{\al})$ with $\al$ totally positive.
In Proposition \ref{prop:kum} we must replace the condition $\ov{u}\in S[N]$
by $\ov{u}\in S^+(k)$ and require $\mfa$ to represent a square in $\Cl^+(k)$, 
and in the beginning of the computation of
$\Phi_k(s)$ after that proposition we must similarly replace all
occurrences of $S[N]$ by $S^+(k)$.

In definitions \ref{def_zkc} and \ref{defsa4} we must add superscripts $\ ^+$ 
to all the groups which are defined, adding everywhere the condition that 
$\be$ or $u$ is totally positive (in the language of class field theory, we write
$\be\equiv1\pmod{^*\c^2S_\infty}$, where $S_\infty$ is the modulus made of
the three infinite places of $k$). We write $Z^+_{\c}$ for the set of elements
of $(\Z_k/\c^2)^*$ which have a totally positive lift, but it is easily checked that
$Z^+_{\c}=Z_{\c}$ and
$Z^+_{\c}[N]=Z_{\c}[N]$, and hence that $z^+_k(\c)=z_k(\c)$, so we do not need
to compute again this subtle quantity. 

We also check that $S^+_{\c^2}[N]=S^+_{\c^2}$, and the exact sequence corresponding to
\eqref{eqn:exseq} of Proposition \ref{prop:exseq} is thus
\begin{equation}\label{eqn:exseq_plus}
1\LR S^+_{\c^2}(k)\LR S^+(k)\LR \dfrac{Z_{\c}[N]}{Z_{\c}^2}\LR C^+_{\c^2}\LR C^+_{(1)}\LR1\;.
\end{equation}
We have the exact sequence 
$1\LR U^+(k)/U^2(k)\LR S^+(k)\LR \Cl(k)[2]\LR1$
(the last $\Cl(k)[2]$ does \emph{not} have a $+$ subscript), and the
ray class group exact sequence shows that
$|\Cl^+(k)/\Cl^+(k)^2|=|\Cl(k)/\Cl(k)^2||U^+(k)/U^2(k)|$, which implies
that 
\begin{equation}\label{eqn:s_cl_eq}
|S^+(k)|=|\Cl^+(k)/\Cl^+(k)^2|\;.
\end{equation}
The computations leading to Corollary \ref{corphik} are identical, and
that corollary is thus valid if we replace $\Phi_k(s)$, $S_{\c^2}[N]$, 
$C_{\c^2}$, $X_{\c^2}$ by the corresponding values with $\ ^+$ superscripts.

The proof of Corollary \ref{corsc2} becomes

$$|S^+_{\c^2}(k)||Z_{\c}[N]/Z_{\c}^2||C^+_{(1)}|=|S^+(k)||C^+_{\c^2}|=|\Cl^+(k)/\Cl^+(k)^2||C^+_{\c^2}|=|C^+_{(1)}||C^+_{\c^2}|\;,$$
where we have also extended (3) of Proposition \ref{prop:exseq}; in other words 
$$|S^+_{\c^2}(k)|/|C^+_{\c^2}|=1/|Z_{\c}[N]/Z_{\c}^2|=z_k(\c)/\N(\c)$$
by Lemma \ref{easyzk1}. Putting everything together proves the theorem.
\end{proof}

\begin{remark} As in Theorem \ref{thm_main_sig}, there 
are exactly two differences between the formula of
Theorem \ref{thma4s4plus} and that of Theorem \ref{thma4s4}: we sum
on $\chi\in X^+_{\c^2}$ instead of $\chi\in X_{\c^2}$, and the
factor in front of $1/(a(k)2^{3s})$ is $1$ instead of
$2^{2-r_2(k)}=4$.

Since only the
trivial characters contribute to asymptotics, this implies (as already 
mentioned in \cite{Coh_a4s4}) that the number of totally real quartics
$L\in\FF(k)$ and positive discriminant up to $X$ is asymptotically $1/4$ of the total 
number, independently of the signatures of the fundamental units.
The same is also true for the set of {\itshape all} quartics of signature
(4, 0) and (0, 2), as proved by Bhargava \cite{B_count_quartic}.

\end{remark}
\subsection{The main theorem with signature conditions}
It is readily checked that the maps in \eqref{it_iso} of Proposition \ref{prop_c4}
also yield isomorphisms $\Cl^+_{(n)}(k)/\Cl^+_{(n)}(k)^2 \simeq C_{(n)}^+$ for $n = 1, 4$,
and that Proposition \ref{prop_c4_ext} similarly extends,
so that the sum in Theorem \ref{thma4s4plus} becomes a summation over 
$\calL^*_2(k)$ instead of $\calL_2(k)$.

In place of the remainder of Proposition \ref{prop_c4} we have the following.

\begin{lemma}\label{lem_c4_plus} We have the following equalities:
\begin{equation}\label{eqn:c4_plus1}
\frac{ |C_{(4)}^+|} {|C_{(1)}^+ |} \cdot
\frac{ |C_{(4)} |} {|C_{(1)} |} = 4,
\end{equation}
\begin{equation}\label{eqn:c4_plus2}
|C_{(4)}| = |C_{(1)}^+|.
\end{equation}
Therefore, $|C_{(4)}^+|/|C_{(1)}^+ |$ is equal to $2$ or $4$, depending on whether there does
or does not exist a nonsquare totally positive unit.
\end{lemma}
\begin{proof} 
By \eqref{eqn:exseq_plus}, we have
\begin{equation}
\frac{ |C_{(4)}^+|} {|C_{(1)}^+ |} \cdot
\frac{ |S^+(k) |} {|S_{(4)}^+(k) |} = 4.
\end{equation}
By Proposition \ref{prop:hec} and Lemma \ref{lem_square}, $|S^+(k)|$ and $|S_{(4)}^+(k)|$ are respectively
equal to the number of quadratic extensions of $k$, together with the trivial extension $k/k$, which
are unramified at infinity and whose
discriminants divide $(4)$ and $(1)$ respectively, which implies the equalities
$|S^+(k)| = |\Cl_{(4)}(k)/\Cl_{(4)}(k)^2|$ and 
$|S^+_{(4)}(k)| = |\Cl(k)/\Cl(k)^2|$. Note the reversal of the $4$ and $1$, a consequence of
Hecke's discriminant computation. Also recall from \eqref{eqn:s_cl_eq} 
that
$|S^+(k)| =  |\Cl^+(k)/\Cl^+(k)^2|$.

We immediately obtain \eqref{eqn:c4_plus2} and \eqref{eqn:c4_plus1} from these computations and the isomorphisms
$\Cl_{(n)}(k)/\Cl_{(n)}(k)^2 \simeq C_{(n)}$,
$\Cl^+_{(n)}(k)/\Cl^+_{(n)}(k)^2 \simeq C_{(n)}^+$ for $n = 1, 4$, and the final statement
follows from (\ref{it_size}) of Proposition \ref{prop_c4}.

\end{proof}
The results of Section \ref{sec_xnew} apply equally to this setting. Theorem
\ref{thmcrux} also holds, with the same proof, where we replace $\calL_2(k)$ with
$\calL_2^*(k)$ and $V^+(k)$ with $V(k)$. Analogues of \eqref{it_l2k_count} and the first part of
\eqref{it_l2k_count2}
of Proposition \ref{prop_l2k} hold, with $\calL(k, 1)$, $\calL_2(k)$, $\rk_2(k)$, $C_{(n)}$
replaced with  $\calL^*(k, 1)$, $\calL^*_2(k)$, $\rk^+_2(k)$, $C^+_{(n)}$ respectively. However,
in light of Lemma \ref{lem_c4_plus} the remainder of Proposition \ref{prop_l2k} no longer holds,
and in particular it is not true that at most one of $\calL(k, 4)$, $\calL(k, 16)$, and $\calL_{tr}(k, 64)$
can be nonempty.

Results analogous to those of Section \ref{sec_arith_s4} hold, with the same proofs, except
for the 
first part of Remarks \ref{rem_sec7},
where the conditions on the valuations of $\alpha \in W^+(k)$ appearing in the various cases are no longer mutually exclusive.

Finally, it is reasonably straightforward to adapt the arguments of Section \ref{sec_quartic_proofs}.
The key step is that we require, for each $\c | (2)$, an isomorphism
\begin{equation}\label{eqn_4c_iso_plus2}
\Cl^+_{\c'^2}(k) / \Cl^+_{\c'^2}(k)^2
\xrightarrow{\phi}
C^+_{\c^2}
\end{equation}
for appropriate $\c' | (2)$ with $\phi(\mfa) = \mfa/\N(\mfa)$. We already 
observed that the proof of Proposition 
\ref{prop_c4} yields such an isomorphism for $\c = \c' = (2)$ and $\c = \c' = (1)$,
and the proof of Proposition \ref{prop_c4_ext} also gives an isomorphism \eqref{eqn_4c_iso_plus2}
in the special cases described there.

For the remaining cases, it suffices to find $\c'$ such that both sides of \eqref{eqn_4c_iso_plus2}
have the same size, as then the desired isomorphism will automatically be induced
by the isomorphism 
$\Cl^+_{(4)}(k) / \Cl^+_{(4)}(k)^2
\xrightarrow{\phi}
C^+_{(4)}$. We lose the property $\phi(\mfa) = \mfa/\N(\mfa)$ for $\mfa$ not coprime to $(2)$,
but as in our analysis without signature conditions, this will not matter.

If  $|C^+_{(4)}| =2 |C^+_{(1)}|$ then nothing changes. If we have
 $|C^+_{(4)}| = 4 |C^+_{(1)}|$ then we need to reinterpret Proposition \ref{prop_list_c}:
 for the splitting types $(1^2 1)$, $(1^3)$, and $(21)$, the first bulleted point in the proof
  shows that $|C^+_{\c^2}| = 2 |C^+_{(1)}|$ for
 some $\c$ between $(1)$ and $(2)$. The second bulleted point in turn
 implies that two of 
 $\calL^*(k, 4)$, $\calL^*(k, 16)$, and $\calL^*_{tr}(k, 64)$ 
 are nonempty, as we previously observed was possible in Proposition \ref{propdist}.
 When $2$ is $(111)$, then only $\calL^*(k, 16)$
is nonempty, but all three primes $\mfp_i$ are `distinguished' for different fields in
$\calL(k, 16)$; we have  $|C^+_{\c^2}| = 2 |C^+_{(1)}|$ when $\c$ is prime, 
and  $|C^+_{(4)}| = 4 |C^+_{(1)}|$
when $\c$ is a product of two or all three primes.

Therefore,  when $|C^+_{(4)}| = 4 |C^+_{(1)}|$ and $2$ has splitting type $(1^2 1)$, $(1^3)$, or $(2 1)$,
two of the cases in Proposition \ref{prop_list_c} occur, and we interpret the proposition as saying that 
$|C^+_{\c^2}| = 2 |C^+_{(1)}|$ for $\c$ appearing for one of them, and 
$|C^+_{\c^2}| = 4 |C^+_{(1)}|$ for $\c$ appearing for both of them. 

The upshot is that our previous association of a set of $\c$ for each field in $\calL_2(k)$ remains accurate
for $\calL^*_2(k)$.
The reinterpreted Proposition \ref{prop_list_c} tells us which $X^+_{\c^2}$ contribute for each field in
$\calL^*_2(k)$, and the remainder of the proof is the same, including all of our computations of
$M_{2, L}$, so that Theorem \ref{thm_main_sig}
follows from Theorem \ref{thma4s4plus} in the same way that 
Theorem \ref{thm_main}
followed from Theorem \ref{thma4s4}.

Some explicit numerical examples are worked out in Section \ref{sec_computations}.

\section{Numerical Computations}\label{sec_computations}

We finish the paper by presenting some numerical examples in a few cases representative of our main results.
The proofs are immediate; the explicit number fields described below may be looked up in databases such as
database \cite{JRnf, parietal}.

\subsection{Examples for the $A_4$-quartic case}\label{exa4}\hfill
\noindent
\begin{itemize}
\item {\bf $\rk_2(k)=0$ and $2$ inert in $k$:}

\smallskip

\noindent
$k$ cyclic cubic of discriminant $49=7^2$, defined by $x^3-x^2-2x+1=0$.

$$\Phi_k(s)=\dfrac{1}{3}\left(1+\dfrac{3}{2^{3s}}\right)\prod_{p\Z_k=\p_1\p_2\p_3}\left(1+\dfrac{3}{p^s}\right)=\dfrac{1}{3}\left(1+\dfrac{3}{2^{3s}}\right)\prod_{p\equiv\pm1\pmod7}\left(1+\dfrac{3}{p^s}\right)\;.$$

The second equality comes from the fact that we are
in an Abelian situation, so such equalities also exist in the
subsequent formulas.

\medskip

\noindent
\item {\bf $\rk_2(k)=2$ and $2$ inert in $k$:}

\smallskip

\noindent
$k$ cyclic cubic of discriminant $26569=163^2$, defined by $x^3-x^2-54x+169=0$.

\[
\Phi_k(s)=\dfrac{1}{3}\left(1+\dfrac{3}{2^{3s}}\right)\prod_{p\Z_k=\p_1\p_2\p_3}\left(1+\dfrac{3}{p^s}\right)\\
+\left(1+\dfrac{3}{2^{3s}}\right)\prod_{p\Z_k=\p_1\p_2\p_3}\left(1+\dfrac{\om_L(p)}{p^s}\right)\;,
\]
where $L$ is the unique $A_4$-quartic field with cubic resolvent $k$, defined
by $x^4-x^3-7x^2+2x+9=0$.

\noindent
\smallskip
\item {\bf $\rk_2(k)=4$ and $2$ totally split in $k$:}

\smallskip

\noindent
$k$ cyclic cubic of discriminant $1019077929=31923^2$, defined by 
$x^3-10641x-227008=0$.

\begin{align*}\Phi_k(s)&=\dfrac{1}{3}\left(1+\dfrac{3}{2^{2s}}+\dfrac{6}{2^{3s}}+\dfrac{6}{2^{4s}}\right)\prod_{p\Z_k=\p_1\p_2\p_3,\ p\ne2}\left(1+\dfrac{3}{p^s}\right)\\
&\phantom{=}+\left(1+\dfrac{3}{2^{2s}}+\dfrac{6}{2^{3s}}+\dfrac{6}{2^{4s}}\right)\prod_{p\Z_k=\p_1\p_2\p_3,\ p\ne2}\left(1+\dfrac{\om_{L_1}(p)}{p^s}\right)\\
&\phantom{=}+\left(1+\dfrac{3}{2^{2s}}-\dfrac{2}{2^{3s}}-\dfrac{2}{2^{4s}}\right)\sum_{2\le i\le 5}\prod_{p\Z_k=\p_1\p_2\p_3,\ p\ne2}\left(1+\dfrac{\om_{L_i}(p)}{p^s}\right)\;,\end{align*}
where the $L_i$ are the five $A_4$-quartic fields with cubic resolvent $k$, 
defined by the respective equations $x^4-2x^3-279x^2-1276x+2132$, 
$x^4-2x^3-207x^2-108x+4464$, $x^4-2x^3-201x^2+154x+4537$,
$x^4-2x^3-255x^2-40x+13223$, $x^4-x^3-237x^2+132x+13908$,
and $L_1$ is distinguished by being the only one of the five fields in which
$2$ is totally split.
\end{itemize}

\subsection{Examples for the $S_4$-quartic case}\label{exs4}

\begin{itemize}
\item $k$ defined by 
$x^3-x^2-3x+1=0$, $\Disc(k) = 148$, $\rk_2^+(k)=0$, and $2$ splits as $(1^3)$ in $k$.

$$\Phi_k(s)=\left(1+\dfrac{1}{2^s}+\dfrac{2}{2^{3s}}\right)
\prod_{p\Z_k=\p_1\p_2}\left(1+\dfrac{1}{p^s}\right)\prod_{p\Z_k=\p_1^2\p_2}\left(1+\dfrac{1}{p^s}\right)\prod_{p\Z_k=\p_1\p_2\p_3}\left(1+\dfrac{3}{p^s}\right)\;.$$

\item $k$ defined by 
$x^3-4x-1=0$, $\Disc(k) = 229$, $\rk_2^+(k)=1$, and $2$ splits as $(21)$ in $k$.

\begin{align*}\Phi_k(s)&=\left(1+\dfrac{1}{2^{2s}}+\dfrac{4}{2^{3s}}+\dfrac{2}{2^{4s}}\right)\prod_{p\Z_k=\p_1\p_2,\ p\ne2}\left(1+\dfrac{1}{p^s}\right)\prod_{p\Z_k=\p_1^2\p_2}\left(1+\dfrac{1}{p^s}\right)\prod_{p\Z_k=\p_1\p_2\p_3}\left(1+\dfrac{3}{p^s}\right)\\
&\phantom{=}+\left(1-\dfrac{1}{2^{2s}}\right)\prod_p\left(1+\dfrac{\om_L(p)}{p^s}\right)\;,\end{align*}
where $L$ is the $S_4$-quartic field of discriminant $64\cdot229$ defined by
$x^4-2x^3-4x^2+4x+2=0$.

\end{itemize}
\subsection{Examples with signature conditions}\label{exsig}
Finally, we work out examples of the series $\Phi_k^+(s)$, described in Section \ref{sec_sig},
in both the $S_4$ and $A_4$ cases. 

\noindent
\begin{itemize}

\item $k$ noncyclic cubic defined by 
$x^3-x^2-5x+4=0$, $\Disc(k) = 469$, where $2$ splits as $(2 1)$.

\begin{align*}
\Phi_k^+(s) = \frac{1}{4} & \Bigg( \bigg(1 + \frac{1}{2^{2s}} + \frac{4}{2^{3s}} + \frac{2}{2^{4s}} \bigg)
\prod_{p\Z_k=\p_1\p_2,\ p\ne2}\left(1+\dfrac{1}{p^s}\right)\prod_{p\Z_k=\p_1^2\p_2,\ p\ne2}\left(1+\dfrac{1}{p^s}\right)
\prod_{p\Z_k=\p_1\p_2\p_3,\ p\ne2}\left(1+\dfrac{3}{p^s}\right)
\\ & + \bigg(1 + \frac{1}{2^{2s}} - \frac{4}{2^{3s}} + \frac{2}{2^{4s}} \bigg) \prod_{\p \neq 2} 
\bigg(1 + \frac{\omega_{L_1}(p)}{p^s} \bigg)
\\ & + \bigg(1 - \frac{1}{2^{2s}} \bigg) \prod_{\p \neq 2} \bigg(1 + \frac{\omega_{L_2}(p)}{p^s} \bigg)
+ \bigg(1 - \frac{1}{2^{2s}} \bigg) \prod_{\p \neq 2} \bigg(1 + \frac{\omega_{L_3}(p)}{p^s} \bigg)
\Bigg)\;.
\end{align*}
In the above, we have $|C^+_{(4)}| = 4|C^+_{(1)}| = 4$, and
there are three fields $L_1, L_2, L_3$ in $\calL_2^*(k)$,
of discriminants $(-1)^2 2^4 \cdot 469$, 
 $(-1)^2 2^6 \cdot 469$, 
and $(-1)^2 2^6 \cdot 469$ respectively, where $(-1)^2$ indicates that each field has two 
pairs of complex embeddings, 
and in which the splitting of $2$
is respectively $(2^2)$, $(1^4)$, and $(1^4)$.
The total of the $2$-adic factors is $1 + 2^{-4s}$, corresponding to the fact that the only totally real quartic field
of discriminant $2^a \cdot 469$ is $\Q(x)/(x^4 - 14 x^2 - 4x + 38)$, of discriminant
$2^4 \cdot 469$.
\medskip
\item
An example where $|C^+_{(4)}| = 2 |C^+_{(1)}|$ is furnished by 
$k := \Q(x)/(x^3 - 4x - 1)$, the unique cubic field of discriminant 229. We again have
three fields in $\calL^*_2(k)$, with discriminants $(-1)^2 229$,
$(-1)^0 2^6 229$, and $(-1)^2 2^6 229$.
Again $\Phi^+_k(s)$ is a sum of four terms which can be written down as in the previous example.
\medskip
\item
An $A_4$ example: $k$ cyclic of discriminant $31^2$, defined by
$x^3-x^2-10x+8=0$, in which $2$ is totally split. 
\begin{multline*}
\ \ \ \ \ \Phi_k^+(s) =
\dfrac{1}{4}\Bigg( \frac{1}{3} \bigg(1 + \frac{3}{2^{2s}} + \frac{6}{2^{3s}} + \frac{6}{2^{4s}} \bigg) 
\prod_{p\Z_k=\p_1\p_2\p_3}\left(1+\dfrac{3}{p^s}\right)
\\ + 
\bigg(1 - \frac{1}{2^{2s}} + \frac{2}{2^{3s}} - \frac{2}{2^{4s}} \bigg)
\prod_{p\Z_k=\p_1\p_2\p_3}\left(1+\dfrac{\om_L(p)}{p^s}\right)\Bigg)\;. \ \ \ \ \ 
\end{multline*}


\end{itemize}


\begin{thebibliography}{99}

\bibitem{AF} J. Armitage and A. Fr\"ohlich,
\emph{Class numbers and unit signatures},
Mathematika \textbf{14} (1967), 94--98.

\bibitem{B2} M.~Bhargava,
\emph{Higher composition laws II: On cubic analogues of Gauss composition},
Ann.~Math. (2) {\bf 159} (2004), no. 2, 865--886.

\bibitem{B3} M.~Bhargava,
\emph{Higher composition laws III: The parametrization of quartic rings},
Ann.~Math. (2) {\bf 159} (2004), no. 3, 1329--1360.

\bibitem{B_count_quartic} M.~Bhargava,
\emph{The density of discriminants of quartic rings and fields},
Ann.~Math. (2) {\bf 162} (2005), no. 2, 1031--1063.


\bibitem{Coh} H.~Cohen,
\emph{Advanced topics in computational number theory}, Graduate Texts in
Math.~{\bf 193}, Springer-Verlag, New York, 1999.

\bibitem{Coh_a4s4} H.~Cohen, {\it Counting $A_4$ and $S_4$ number fields with
given resolvent cubic\/}, Proceedings Banff Conference in honor of
Hugh Williams, Fields Inst.~Comm.~{\bf 41} (2004), 159--168.

\bibitem{CDO_C4} H.~Cohen, F.~Diaz y Diaz, and M.~Olivier,
{\it Counting cyclic quartic extensions of a number field\/},
J.~Th.~Nombres Bordeaux {\bf 17} (2005), 475--510.

\bibitem{CDO_D4} H.~Cohen, F.~Diaz y Diaz, and M.~Olivier,
{\it Enumerating quartic dihedral extensions of $\Q$\/},
Compositio Math. {\bf 133} (2002), 65--93.

\bibitem{CDO_V4} H.~Cohen, F.~Diaz y Diaz, and M.~Olivier,
{\it Counting biquadratic extensions of a number field\/},
preprint.

\bibitem{CDO_all} H.~Cohen, F.~Diaz y Diaz, and M.~Olivier,
{\it Counting discriminants of number fields\/},
J.~Th.~Nombres Bordeaux {\bf 18} (2006), 573--593.

\bibitem{CDO_quartic} H.~Cohen, F.~Diaz y Diaz, and M.~Olivier,
{\it Counting $A_4$ and $S_4$ extensions of number fields\/}, unpublished
preprint.

\bibitem{CM} H. Cohen and A. Morra,
\emph{Counting cubic extensions with given quadratic resolvent}, J. Algebra \textbf{325} (2011), 461--478. 
(Theorem numbers refer to the published version, which are different than in the arXiv version.)

\bibitem{CT3} H.~Cohen and F.~Thorne, {\it Dirichlet series associated 
to cubic fields with given quadratic resolvent\/}, preprint,
available at \url{http://arxiv.org/abs/1301.3563}.

\bibitem{cohn} H. Cohn,
\emph{The density of abelian cubic fields}, 
Proc. Amer. Math. Soc. \textbf{5} (1954), 476--477.

\bibitem{H} H.~Hasse,
\emph{Number theory}, 
Springer-Verlag, Berlin, 2002.

\bibitem{Jeh} A.~Jehanne, \emph{R\'ealisation sur $\Q$ de corps de degr\'es
$6$ et $8$}, PhD thesis, Universit\'e Bordeaux I (1993).

\bibitem{JR} J.~Jones and D.~Roberts,
\emph{A database of local fields},
J.~Symbolic Comput.~\textbf{41} (2006), no.~1, 80--97; accompanying database
available online at \url{http://math.asu.edu/~jj/localfields/}.

\bibitem{JRnf} J.~Jones and D.~Roberts,
\emph{A database of number fields},
\url{http://hobbes.la.asu.edu/NFDB/}.

\bibitem{Mar} J.~Martinet, \emph{Quartic fields}, unpublished preprint.

\bibitem{parietal} The PARI Group, J. Voight, J. Jones, and D. Roberts,
\emph{Global number fields}, online databases available at \url{http://www.lmfdb.org/NumberField/} and
\url{http://pari.math.u-bordeaux.fr/pub/numberfields/}.


\end{thebibliography}
\end{document}